\documentclass[review,hidelinks,onefignum,onetabnum]{siamart220329}

\usepackage{lipsum}
\usepackage{amsfonts}
\usepackage{graphicx}
\usepackage{epstopdf}
\ifpdf
  \DeclareGraphicsExtensions{.eps,.pdf,.png,.jpg}
\else
  \DeclareGraphicsExtensions{.eps}
\fi

\newsiamremark{remark}{Remark}
\newsiamremark{hypothesis}{Hypothesis}
\crefname{hypothesis}{Hypothesis}{Hypotheses}
\newsiamthm{claim}{Claim}
\newsiamthm{assumption}{Assumption}

\headers{Quadratic Optimization Approach to DCI}{K.~Bergstrom, T.~Butler, and T.~Wildey}

\title{A Distributions-based Approach to Data-Consistent Inversion}

\author{Kirana O.~Bergstrom\thanks{Department of Mathematical and Statistical Sciences, University of Colorado - Denver, Denver, CO 
  (\email{kirana.bergstrom@ucdenver.edu}, \url{http://www.imag.com/\string~ddoe/}).}
\and Troy D.~Butler \thanks{Department of Mathematical and Statistical Sciences, University of Colorado - Denver, Denver, CO 
  (\email{troy.butler@ucdenver.edu}).}
\and Tim Wildey\thanks{Computational Mathematics Department, Center for Computing Research, Sandia National Labs, Albuquerque, NM. 
Sandia National Laboratories is a multimission laboratory managed and operated by National Technology \& Engineering Solutions of Sandia, LLC, a wholly owned subsidiary of Honeywell International Inc., for the U.S. Department of Energy’s National Nuclear Security Administration under contract DE-NA0003525.
  (\email{tmwilde@sandia.gov})}
}

\usepackage{amsopn}

\ifpdf
\hypersetup{
  pdftitle={A Distributions-based Approach for Data-Consistent Inversion},
  pdfauthor={Authors here}
}
\fi

\usepackage{bm}
\usepackage{subfig}
\usepackage{xcolor}
\usepackage{diagbox}
\usepackage[normalem]{ulem}
\usepackage{algorithmic, tabularx}
\usepackage{enumerate}
\usepackage{verbatim}
\makeatletter
\newcommand{\multiline}[1]{%
    \begin{tabularx}{\dimexpr\linewidth-\ALG@thistlm}[t]{@{}X@{}}
        #1
    \end{tabularx}
}

\newenvironment{routine}[1][htb]{%
  \renewcommand{\ALG@name}{Routine}
  \begin{algorithm}[#1]%
  }{\end{algorithm}
}

\newcommand{\abs}[1]{\left\vert#1\right\vert}

\newcommand{\set}[1]{\left\{#1\right\}}
\newcommand{\seq}[1]{\left\{#1\right\}_{n\in\mathbb{N}}}

\newcommand{\pspace}{\Lambda}
\newcommand{\dspace}{\mathcal{D}}

\newcommand{\dmeas}{\mu_{\dspace}}
\newcommand{\pborel}{\mathcal{B}_{\pspace}}
\newcommand{\dborel}{\mathcal{B}_{\dspace}}

\newcommand{\predsamples}{\{\bm{q}^i\}_{i=1}^n}
\newcommand{\initsamples}{\{\bm{\lambda}^i\}_{i=1}^n}
\newcommand{\weightset}[1]{\{#1_i\}_{i=1}^n}

\newcommand{\preddens}{\pi_{\text{pred}}}
\newcommand{\obsdens}{\pi_{\text{obs}}}
\newcommand{\initdens}{\pi_{\text{init}}}
\newcommand{\updens}{\pi_{\text{update}}}

\newcommand{\predmeas}{P_{\text{pred}}}
\newcommand{\obsmeas}{P_{\text{obs}}}
\newcommand{\initmeas}{P_{\text{init}}}
\newcommand{\upmeas}{P_{\text{update}}}

\newcommand{\predcdf}{F_{\text{pred}}}
\newcommand{\obscdf}{F_{\text{obs}}}
\newcommand{\initcdf}{F_{\text{init}}}
\newcommand{\upcdf}{F_{\text{update}}}

\newcommand{\prededf}{F_{\text{pred}}^n}

\newcommand{\predwedf}{F_{\text{pred};\bm{w}}^n}

\newcommand{\initwedf}{F_{\text{init};\bm{w}}^n}

\newcommand{\initwemeas}{P_{\text{init};\bm{w}}^n}

\newcommand{\predwedfpart}{F_{\text{pred};\bm{u}}^{n,p}}

\newcommand{\initwedfpart}{F_{\text{init};\bm{u}}^{n,p}}

\newcommand{\predwemeaspart}{P_{\text{pred};\bm{u}}^{n,p}}

\newcommand{\initwemeaspart}{P_{\text{init};\bm{u}}^{n,p}}

\begin{document}

\maketitle

\begin{abstract}
We formulate a novel approach to solve a class of stochastic problems, referred to as data-consistent inverse (DCI) problems, which involve the characterization of a probability measure on the parameters of a computational model whose subsequent push-forward matches an observed probability measure on specified quantities of interest (QoI) typically associated with the outputs from the computational model.
Whereas prior DCI solution methodologies focused on either constructing non-parametric estimates of the densities or the probabilities of events associated with the pre-image of the QoI map, we develop and analyze a constrained quadratic optimization approach based on estimating push-forward measures using weighted empirical distribution functions.
The method proposed here is more suitable for low-data regimes or high-dimensional problems than the density-based method, as well as for problems where the probability measure does not admit a density.
Numerical examples are included to demonstrate the performance of the method and to compare with the density-based approach where applicable.
\end{abstract}



\section{Introduction}
The formulation and solution of an inverse problem impacts the type of information revealed about parameters of a model from observed data.
For instance, inverse problems can be posed deterministically and solved using optimization methods~\cite{vogel,kirsch} to determine ``parameters of best fit'' in a particular sense.
Some stochastic inverse problems seek a quantification of epistemic uncertainty in likely parameter values utilizing Bayesian methods~\cite{stuart_2010,idier}.
In this work, we seek to quantify aleatoric uncertainties about the parameters of the model from probabilistic information of the observed data.
In other words, we seek a probability measure of the model inputs (i.e., the model parameters) given a probability measure that characterizes uncertainties of the model outputs.
We require the solution to this inverse problem to be \textit{data-consistent}, meaning that the push-forward of the probability measure on the parameter space matches a given target distribution on the observed data space.
This type of inverse problem naturally arises in scenarios where observed variability in data is primarily due to intrinsic variability in the model inputs.

\subsection{Comparison to related inverse problem work}
A probability distribution can be characterized by its cumulative distribution function (CDF), measure, or, if it exists, its density (i.e., its Radon-Nikodym derivative).
Previous approaches to solving the data-consistent inverse (DCI) problem either approximate a pullback measure directly through event approximation in both the input and output spaces as in~\cite{measures} or use a non-parametric kernel density approximation in the output space~\cite{Butler2018}.
These approaches are challenging to implement when the number of simulated data are limited, e.g., when the computational model is expensive to evaluate.
Moreover, there is an implicit assumption that sufficient information exists on observations to specify either a measure or density.
The method described in this paper circumvents these issues by directly approximating a CDF via an optimally-weighted empirical distribution function (EDF).
The proposed method does not require the existence of densities for any of the distributions involved, and it does not rely on kernel density estimation methods, which may be unreliable in low-data regimes and are computationally expensive in high-dimensional spaces.

\subsection{Contributions}
The contributions of this work are both algorithmic and theoretical.
The method we propose builds upon the approach introduced and analyzed in~\cite{Amaral2017,doi:10.2514/6.2019-0967} for approximating push-forward measures by performing a change-of-measure objective using EDFs via solution to an optimization problem that determines optimal weights on the simulated output samples.
These weights minimize the $L^2$-norm between the weighted EDF and a given target distribution function and are guaranteed to exist since they are the result of a strictly convex quadratic optimization problem.
Thus, for the case where simulated data are limited, a solution is achievable that is guaranteed optimal in an $L^2$-sense.
Moreover, the target distribution may be in the form of either a CDF or an EDF, which guarantees a solution to this approach even in cases further limited by availability of observational data.

The key to building upon the optimization-based approach for push-forward EDFs to solve the DCI problem is through the addition of a critical binning step in the output space. 
This provides a practical partitioning of the observed space prior to solution of the optimization problem that permits the proper distribution of weights in the input space.
We both prove and demonstrate numerical convergence of the optimization-based solution to the DCI solution.

\subsection{Organization}
Background and previous approaches for DCI are described in~\Cref{sec:dci}.
\Cref{sec:cdfs} describes the optimization-based approach.
Convergence and other theoretical results are discussed in~\Cref{sec:theory} (proofs are found in the appendices). 
\Cref{sec:apps} contains applications and examples.
Conclusions and future directions follow in~\Cref{sec:conclusions}.
\Cref{sec:code} provides details on obtaining the code utilized to produce the results in this work and~\Cref{sec:ack} contains acknowledgements.

\section{Data-Consistent Inversion (DCI)}\label{sec:dci}
We summarize the concepts of data-consistent inversion (DCI), and direct the interested reader to~\cite{Butler2018} for more details.

\subsection{Terminology and Notation}

Let $\bm{\lambda}\in\Lambda$ denote the uncertain parameters in a model for which the goal is to estimate a distribution given a particular distribution on Quantities of Interest (QoI) corresponding to model outputs. 
Let $Q$ denote the QoI map from the parameter space, $\pspace$, to the data space, $\dspace$, i.e., $Q(\pspace)=\dspace$.
We equip $\Lambda$ and $\mathcal{D}$ with $\sigma$-algebras $\mathcal{B}_\Lambda$ and $\mathcal{B}_{\mathcal{D}}$ as well as measures $\mu_\Lambda$ and $\mu_{\mathcal{D}}$, respectively.
These measures are dominating measures that allow us to express probability measures defined on these spaces as Radon-Nikodym derivatives. 
While not necessary, it is often the case that $\Lambda\subset\mathbb{R}^p$ and $\mathcal{D}\subset\mathbb{R}^d$, $\mathcal{B}_\Lambda$ and $\mathcal{B}_\mathcal{D}$ are Borel $\sigma$-algebras, and $\mu_\Lambda$ and $\mu_\mathcal{D}$ are Lebesgue measures.
In that case, the Radon-Nikodym derivatives are referred to as probability density functions or more simply as densities.

Due to the variability in the parameters, the QoI follow a distribution.
The distribution of these observations defines the \textit{observed distribution}, and corresponds to a probability measure, $\obsmeas$, on the measurable space $(\mathcal{D}, \mathcal{B}_\mathcal{D})$.
The problem of determining a distribution of parameters that could have produced the observed distribution is a stochastic inverse problem that that can be solved using DCI.
In mathematical terms, the problem is stated as; given an observed measure $\obsmeas$ on $(\mathcal{D}, \mathcal{B}_\mathcal{D})$, we seek a measure $P_\Lambda$ on $(\Lambda, \mathcal{B}_\Lambda)$ such that for $A\in \mathcal{B}_\mathcal{D}$,
\begin{equation}\label{eq:PF-measure}
P_\Lambda(Q^{-1}(A)) = \obsmeas(A).
\end{equation}
Such a measure is $P_\Lambda$ is called a \textit{pullback} of $\obsmeas$, and $\obsmeas$ is the \textit{push-forward} of $P_\Lambda$.
Note that we are using the common measure-theoretic shorthand notation $Q^{-1}$ to denote the \textit{pre-image} map (i.e., we are not assuming $Q$ is invertible). 

We emphasize that the stochastic inverse problem is often ill-posed, i.e., many pullback measures may exist.
Even in cases where the dimension of $\mathcal{D}$ equals or even exceeds the dimension of $\Lambda$, nonlinear $Q$ will often result in non-uniqueness of solutions.
Section~\ref{sec:density} below summarizes both a theoretical and practical approach for constructing a particular solution under some additional assumptions.

\subsection{Density-Based Solutions}
\label{sec:density}
Here, we summarize the theoretical construction of the ``density-based'' solution in the more general context of Radon-Nikodym derivatives of probability measures.
We use the term ``density-based'' because it is frequently the case that the derivatives are densities as stated above.

If $\obsmeas$ admits a Radon-Nikodym derivative with respect to $\mu_\mathcal{D}$, it is denoted $\obsdens$.
We then re-frame the stochastic inverse problem in terms of Radon-Nikodym derivatives where~\eqref{eq:PF-measure} is rewritten as
\begin{equation}\label{eq:PF-densities}
P_\Lambda(Q^{-1}(A)) = \int_{Q^{-1}(A)} \pi_\Lambda \, d\mu_\Lambda = \int_A \obsdens\, d\mu_\mathcal{D}= \obsmeas(A).
\end{equation}
In other words, the goal is to determine a pullback of $\obsmeas$ in terms of a Radon-Nikodym derivative with respect to $\mu_\Lambda$, which we denote here by $\pi_\Lambda$.

Clearly, \eqref{eq:PF-measure} dictates some restrictions on the structure of a solution, $\pi_\Lambda$, in terms of its aggregate behavior on sets defined by $Q^{-1}$, but this QoI map provides no information as to how $\pi_\Lambda$ should behave \textit{within} a set $Q^{-1}(\bm{q})$ for a given $\bm{q}\in\mathcal{D}$.

The key theoretical tool for constructing a solution to the stochastic inverse problem is the disintegration theorem~\cite{disintegration}.
Below, we summarize a version for probability measures that describes the \textit{unique} decomposition of any probability measure defined on $(\Lambda,\mathcal{B}_\Lambda)$ in terms of its associated push-forward and a family of conditional probability measures. 

\begin{theorem}\label{thm:disintegration}[Disintegration Theorem] Assume $Q : \Lambda  \rightarrow  \mathcal{D}$  is $\mathcal{B}_\Lambda$-measurable, $P_\Lambda$  is a probability measure on $(\Lambda ,\mathcal{B}_\Lambda)$, and $P_\mathcal{D}$  is the push-forward measure of $P_\Lambda$  on $(\mathcal{D} ,\mathcal{B}_\mathcal{D})$. There exists a $P_\mathcal{D}$-almost everywhere uniquely defined family of conditional probability measures $\{ P_{\bm{q}}\}_{{\bm{q}}\in \mathcal{D}}$ on $(\Lambda ,\mathcal{B}_\Lambda)$ such that for any $A \in \mathcal{B}_\Lambda$,
\begin{equation*}
P_{\bm{q}}(A) = P_{\bm{q}}(A \cap  Q^{-1}(\bm{q})),
\end{equation*}
so $P_{\bm{q}}(\Lambda  \setminus  Q^{-1}(\bm{q})) = 0$, and there exists the following disintegration of $P_\Lambda$, 
\begin{equation*}
P_\Lambda (A) = \int_\mathcal{D}P_{\bm{q}}(A) dP_\mathcal{D} (\bm{q}) = \int_\mathcal{D}\left(\int_{A\cap Q^{-1}(\bm{q})}dP_{\bm{q}}(\bm{\lambda}) \right)dP_\mathcal{D} (\bm{q}),
\end{equation*}
for $A \in \mathcal{B}_\Lambda$.
\end{theorem}

From Theorem~\ref{thm:disintegration}, it is self-evident that $P_\Lambda$ is a solution to the stochastic inverse problem if its push-forward $P_\mathcal{D}$ is equal to $\obsmeas$. 
While this gives some mechanism for checking if a proposed probability measure is a solution to the stochastic inverse problem, it fails to address how to construct $P_\Lambda$ since the family of conditional probability measures remains undefined. 
To address this, we utilize an initial probability measure $\initmeas$ on $(\Lambda, \mathcal{B}_\Lambda)$.
The disintegration of this measure is used to construct the unknown family of conditional probability measures.
The push-forward of $\initmeas$ defines a measure we refer to as the \textit{predicted} probability measure, $\predmeas$.

In \cite{Butler2018}, it is shown that if the initial, observed, and predicted distributions admit Radon-Nikodym derivatives, and if the observed measure $\obsmeas$ is absolutely continuous with respect to the predicted measure $\predmeas$, then Theorem~\ref{thm:disintegration} implies a solution to the stochastic inverse problem is given by
\begin{equation}\label{eq:update}
\upmeas(A) := \int_{\mathcal{D}} \left(\int_{A\cap Q^{-1}(\bm{q})} \frac{\initdens(\bm{\lambda})}{\preddens(Q(\bm{\lambda}))}d\mu_{\Lambda,\bm{q}}\right)\obsdens(\bm{q})d\mu_\mathcal{D},
\end{equation}
for all $A \in \mathcal{B}_\Lambda$, where $\{\mu_{\Lambda,{\bm{q}}}\}_{{\bm{q}} \in \mathcal{D}}$ comes from the disintegration of the measure $\mu_\Lambda$, $\preddens$ is the Radon-Nikodym derivative of $\predmeas$, and the term $\initdens(\bm{\lambda})/\preddens(Q(\bm{\lambda}))$ defines the conditional density associated with $Q^{-1}(\bm{q})$ for each $\bm{q}\in\dspace$. 
For further details of this derivation, and proof that such a solution is data-consistent, we refer the reader to \cite{Butler2018}.

The term $\obsdens(\bm{q})$ in~\eqref{eq:update} can be brought into the inner integral by rewriting it as $\obsdens(Q(\bm{\lambda}))$. 
It follows that the data-consistent solution is given by integrating the following Radon-Nikodym derivative over $\Lambda$, 
\begin{equation}\label{eq:density}
\updens(\bm{\lambda}):=\initdens(\bm{\lambda})\frac{\obsdens(Q(\bm{\lambda}))}{\preddens(Q(\bm{\lambda}))} =\initdens(\bm{\lambda}) r(\bm{\lambda})
\end{equation}
where, by a standard result involving changes of measures and Radon-Nikodym derivatives\footnote{cf.~Proposition 3.9(a) in \cite{Folland_book}.}, the ratio $r(\bm{\lambda}) = \frac{\obsdens(Q(\bm{\lambda}))}{\preddens(Q(\bm{\lambda}))}$ defines the Radon-Nikodym derivative of $\obsmeas$ with respect to $\predmeas$ and serves to re-weight the initial likelihoods of parameters.
It is worth noting that for each datum $\bm{q}\in\dspace$ that $r(\bm{\lambda})$ is \textit{constant} for all $\bm{\lambda}\in Q^{-1}(\bm{q})$.
In other words, this re-weighting updates the initially assumed likelihoods $Q^{-1}(\bm{q})$ for each $\bm{q}\in\dspace$ without updating the initially assumed conditional likelihoods of points within such sets. 
It is this observation that leads to the reference of the solution given in~\eqref{eq:update} and~\eqref{eq:density} as an \textit{update} to the initial distribution. 

\subsection{Numerical Approximation of Densities}
Even when spaces are nominally infinite-dimensional, practical considerations and discretizations often result in $\pspace\subset\mathbb{R}^p$ and $\dspace\subset\mathbb{R}^d$ for some finite $p$ and $d$ so that the Radon-Nikodym derivatives (with the exception of $r(\bm{\lambda})$) are densities as previously mentioned.
We are often confronted with the situation where these densities must be approximated from a finite set of (typically iid) samples from the observed distribution along with a finite set of (typically iid) samples from the initial distribution and their associated predicted values. 
Below, we describe a typical approach based on direct estimation of the observed and predicted densities.

Suppose we have $m$ data points $\{\bm{q}^1, \dots, \bm{q}^m\}$, equal to $\{Q(\bm{\lambda}^1)),\dots,Q(\bm{\lambda}^m))\}$, where $\set{\bm{\lambda}^1, \dots, \bm{\lambda}^m}$ are unobserved.
Here, we take a moment to clearly define the use of subscript and superscript indices throughout this work.
Since any collection of samples may denote a collection of vectors, we denote each sample within the collection using superscripts.
Any index subscript denotes a component of a vector; i.e., the $i$th sample of a $d$-dimensional QoI is denoted as $\bm{q}^i = [q_1^i,\dots,q_d^i]^\top$.

Any parametric or non-parametric density-estimation method can be applied to estimate $\obsdens$ from these $m$ samples. 
Some popular non-parametric density-estimation techniques utilized in the literature are Kernel Density Estimation (KDE) \cite{Devroye1985, Chen17}, Normalizing Flows (NFs) \cite{PNR+21}, and Dirichlet-process based mixture models \cite{EW95, Neal00}.

We emphasize that even if the initial density is given exactly, the predicted density is not typically known except for the simplest of QoI maps and initial distributions.
In practice, the predicted density, $\preddens$, is approximated by propagating $n$ iid samples from the initial distribution, $\set{\bm{\lambda}^1,\dots,\bm{\lambda}^n}$, through the map $Q$, to create a set of iid samples from the (unknown) predicted distribution $\{Q(\bm{\lambda}^1)=\bm{q}^1,\dots,Q(\bm{\lambda}^n) = \bm{q}^n\}$ for which any density estimation technique can be applied. 
As previously mentioned, the theoretical existence of $\updens$ assumes $\obsmeas$ is absolutely continuous with respect to $\predmeas$.
However, in practice, we often make a stronger assumption that allows us to utilize this set of iid samples from the initial distribution to create either an estimate of $\updens$ or apply rejection sampling to this set to generate an iid set of samples from the updated distribution.
We refer the interested reader to \cite{Butler2018} for more details, and simply restate the stronger form below.
\begin{assumption}[Strong form of the Predictability Assumption]\label{ass:predictability}
There exists a constant $C > 0$ such that for almost every $\bm{q} \in \mathcal{D}$, $\obsdens(\bm{q}) \le C \preddens(\bm{q})$.
\end{assumption}

To verify this assumption is satisfied in practice, we utilize a quantitative diagnostic that we summarize below.
If the predictability assumption is satisfied, then $\updens$ given in~\eqref{eq:density} defines a density, which implies
\begin{equation}\label{eq:diagnostic}
    1 = \int_\pspace \updens(\bm{\lambda})\, d\mu_\pspace = \int_\pspace \initdens(\bm{\lambda})r(\bm{\lambda})\, d\mu_\pspace =: \mathbb{E}_\text{init} (r(\bm{\lambda})).
\end{equation}
Thus, in practice, we verify the predictability assumption holds by computing the sample average of $r(\bm{\lambda})$ on a random sample drawn from the initial distribution, and comparing this to a value of unity. 


\subsection{An illustrative example}\label{sec:heat_eq_ex}
Consider a manufacturing process for creating metal alloy rods to be used in a high-temperature environment such as welding.
Inconsistencies in the manufacturing process, such as the mechanical processes for the cutting, measuring, and smoothing of the ends of the rods, coupled with inconsistencies in the alloy, results in variations in rod lengths, $\ell$, and thermal diffusivities, $\kappa$.
For this example, assume that $\ell\in [1.9,2.1]$ and $\kappa \in [0.5,1.5]$ are both dimensionless.
Further assume that the objective is to determine the distribution of $\ell$ and $\kappa$ prior to the utilization of the rods in an engineered system whose functional safety is dependent on these parameters.
Using the notation above, we let $\bm{\lambda} = (\ell, \kappa)$ and $\Lambda=[1.9,2.1]\times[0.5,1.5]\subset\mathbb{R}^2$.

We now define a model of an experimental procedure that leads to a QoI map between the measure spaces $(\Lambda, \mathcal{B}_\Lambda, \mu_\Lambda)$ and $(\mathcal{D}, \mathcal{B}_\mathcal{D}, \mu_\mathcal{D})$.
Assuming that the rods are long and thin, the temperature of the rod at spatial location $0<x<\ell$ and time $t>0$ is modeled by the one-dimensional heat equation
\begin{equation*}\left\{
\begin{matrix}
\frac{\partial}{\partial t}u(x,t) = \kappa \frac{\partial^2}{\partial x^2} u(x,t), & 0 < x < \ell,\\
u(0,t) = u(\ell, t) = 0, & t > 0, \\
u(x,0) = x, & 0 < x < \ell,
\end{matrix}
\right.
\end{equation*}
where $u$ is (dimensionless) temperature.
For the sake of illustration and simplicity, we consider homogeneous Dirichlet boundary conditions, where ``$0$'' is the ambient temperature, throughout the experiment.
The initial temperature profile $u(x,0)=x$ is chosen to provide an interesting response over both space-time and the parameter space.
The analytic solution for the temperature is given by
\begin{equation}\label{eq:an_sol}
u(x,t;\bm{\lambda}) = \frac{2\ell^2}{\pi}\sum^\infty_{k=1}\frac{(-1)^{k+1}}{k}\exp\left[-\kappa \frac{k\pi}{\ell^2}t\right]\sin\left(\frac{k\pi x}{\ell}\right).
\end{equation}
Assume for each experiment that we have a single sensor that can accurately and precisely measure the temperature at a specific point $x^\star=1.2$ along the rod and point $t^\star=0.01$ in time.
In this case, $Q(\bm{\lambda}) = u(x^\star,t^\star;\bm{\lambda})$ defines the QoI map from $\Lambda$ to $\mathcal{D}\subset\mathbb{R}$.
For practical computations, we truncate this sum after the $100$th term.
Note that for this example the QoI is defined by a single temperature measurement in space-time so that $\mathcal{D}\subset\mathbb{R}$ whereas the parameter space $\Lambda\subset\mathbb{R}^2$.
Subsequently, there are infinitely many possible combinations of $\ell$ and $\kappa$ in the contour defined by $Q^{-1}(\bm{q})$ for a given datum $\bm{q}\in\mathcal{D}$, and the conditional probability for these $(\ell,\kappa)$ pairs is not specified by~\eqref{eq:PF-measure}.
See Fig.~\ref{fig:contours} for an illustration.
\begin{figure}
\centering
\includegraphics[width=0.6\textwidth]{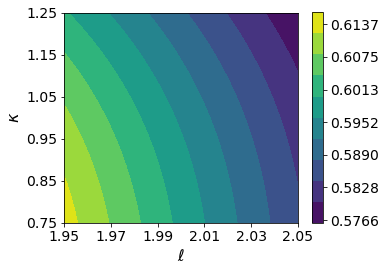}
\caption{Contours for the map $Q$ in the illustrative example. Note the distinct ``contour sets'' whose probabilities are uniquely defined by $\obsmeas$. However, the probabilities \textit{within} these contour sets cannot be determined by $\obsmeas$.}
\label{fig:contours}
\end{figure}
For the sake of simplicity, to generate the observed samples, we assume $\obsdens$ is normal and take $m$ samples from it.
This simulates a typical scenario where we do not know the exact observed distribution but instead have access to samples drawn from this distribution.
Here, we utilize a Gaussian KDE to estimate $\obsdens$ from the data. 

\begin{figure}
\centering
\includegraphics[width=0.8\textwidth]{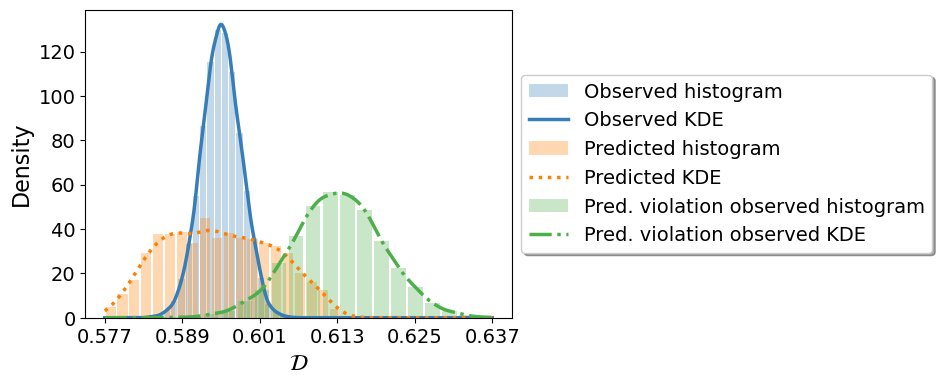}
\caption{Predicted and observed histograms for the illustrative example and their estimated KDEs. We also show an alternate possible observed distribution, labeled ``Pred. violation observed'' because for this example, Assumption~\ref{ass:predictability} is violated: this observed distribution is not absolutely continuous w.r.t. the predicted.}
\label{fig:dists}
\end{figure}

We assume a uniform initial distribution $\pspace$, and propagate $n=2E3$ iid samples from this distribution to produce samples from the predicted distribution.
Since $\dspace$ is dimensionless, we can either visually confirm Assumption~\ref{ass:predictability} by comparing the observed and predicted density approximations in Fig.~\ref{fig:dists}, or compute the diagnostic, which to five significant digits is $\mathbb{E}_\text{init}(r(\bm{\lambda}))\approx 1.0064$. 
We pause here to demonstrate the usefulness of this diagnostic.
For the predictability-violating observed distribution shown in Fig \ref{fig:dists}, the diagnostic is $\mathbb{E}_\text{init}(r(\bm{\lambda}))\approx 0.5118$.
For more complex examples when the distributions are not easily visualized, the diagnostic provides a reliable way to verify if Assumption~\ref{ass:predictability} is satisfied.

With Assumption~\ref{ass:predictability} verified, we move onto interrogating $\updens$ via the weights $r(\bm{\lambda})$ available on the set of initial samples.
In the left plot of Fig.~\ref{fig:dens_method}, we visualize $r(\bm{\lambda})$ as a function on the set of initial samples, while in the right plot we show an iid set of samples for $\updens$ generated via rejection sampling based on these $r(\bm{\lambda})$ values. 

\begin{figure}
\centering
{\includegraphics[height=0.38\textwidth]{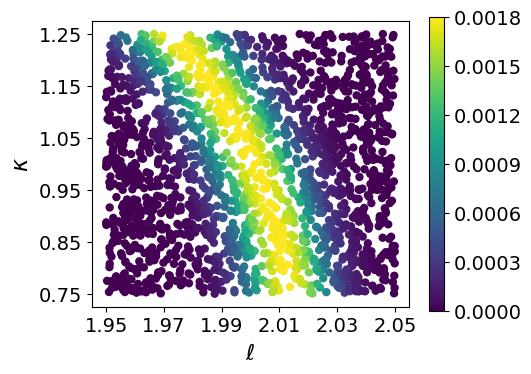}}
{\includegraphics[height=0.38\textwidth]{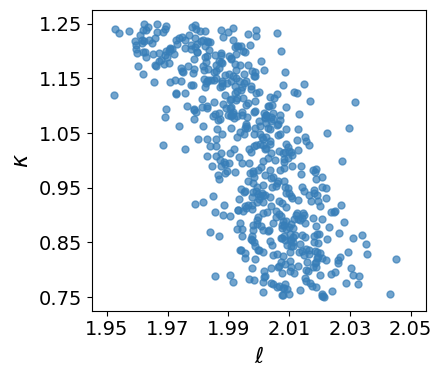}}
\caption{Result of density-based method for data-consistent inversion. On the left we show density-based weights plotted on the initial samples. The result of rejection sampling is shown on the right.}
\label{fig:dens_method}
\end{figure}

\section{Empirical Distributions and Optimal Estimation}\label{sec:cdfs}

The density-based method for solving the DCI problem from Section~\ref{sec:dci} requires the initial, predicted, and observed distributions to all admit densities.
Moreover, when any of these densities are estimated, a sufficient number of samples must exist to approximate them, e.g., using KDEs.
In practice, the number of observed or simulated samples may be limited due to either high experimental or computational cost in obtaining each sample. 
Moreover, utilizing density estimation methods may fail 
if even a single one of the densities (initial, predicted, or observed) fails to exist.
We therefore propose an alternative method based on empirical distribution functions (EDFs), which always exist as approximations to cumulative distribution functions (CDFs) that are in 1-to-1 correspondence with probability measures. 

An attractive feature of the density method is that all of the critical computations take place in the data space, which is generally lower-dimensional than the parameter space.
Specifically, the densities of the predicted and observed distributions are estimated, and a ratio of densities is computed, all in the data space.
While the ratios are then applied to samples in the parameter space, the computations are restricted to the data space.
We pursue a distribution-based approach with a similar feature.
The specific method we propose, in Section~\ref{sec:cdfs_partition}, is motivated from the empirical importance weights for a change of probability measure method described in \cite{Amaral2017}.
In \cite{Amaral2017}, iid samples from an unknown \textit{proposal} distribution are reweighted using an optimization-based method to perform a change-of-measure from the proposal to a \textit{target} distribution that itself may only be given in terms of a set of iid samples.
We summarize this method, using our notation and terminology, in Section~\ref{sec:cdfs_forward}.

\subsection{Computation of optimal weights on the data space}\label{sec:cdfs_forward}

For all but the simplest of QoI maps and initial distributions, the predicted distribution is unknown and must be estimated.
We assume access to a set of $n$ iid samples $\predsamples$ to estimate the predicted empirical CDF as
\begin{equation*}
\prededf(\bm{q}) =\frac{1}{n}\sum^n_{i=1}\mathbb{I}(\bm{q} \preceq \bm{q}^i),
\end{equation*}
which converges almost everywhere to $\predcdf$ as $n \to \infty$, e.g., see Theorem 20.6 in \cite{Billingsley2012}.
Note that because the samples are vectors in $\mathbb{R}^d$ space, we use the symbols $\preceq$ and $\succeq$ to denote element-wise less than and element-wise greater than, respectively.
For example if $\bm{x} \preceq \bm{y}$, then $x_1 \le y_1$, $x_2 \le y_2, \dots, x_d \le y_d$.

An EDF over $n$ samples can be weighted with a set of $n$ scalars $\weightset{w}$ to form a \textit{weighted EDF}, as
\begin{equation}\label{eq:predwedf}
\predwedf(\bm{q}) =\frac{1}{n}\sum^n_{i=1}w_i\mathbb{I}(\bm{q} \preceq \bm{q}^i),
\end{equation}
where we denote $\bm{w} = [w_1,\dots,w_n]^\top$ as the $n$-dimensional vector of the scalar weights with the constraint that $w_i \ge 0$ for $1 \le i \le n$ and $\frac{1}{n}\sum^n_{i=1}w_i = 1$ so that $\predwedf$ defines a valid probability distribution. 

In \cite{Amaral2017}, it is proposed that the weights should minimize the quadratic function
\begin{equation*}
\frac{1}{2}\int_\mathcal{D} \left(\predwedf(\bm{q}) - \obscdf(\bm{q})\right)^2 d\mu_\mathcal{D},
\end{equation*}
which is the square of the $L^2$-norm of the difference between the weighted predicted EDF and the true observed CDF.
Moreover, \cite{Amaral2017} proves that this choice of $\bm{w}$ defines a distribution with a CDF that converges in both $L^1$ and $L^2$ to the observed CDF $\obscdf$.

We note here that, as in \cite{Amaral2017}, we can, without loss of generality, assume that the support of the predicted sample distribution is contained within the $d$-dimensional hypercube. 
If this is not the case, we simply scale a bounding box for the samples from the predicted distribution using either the known (assumed compact) support of distribution, or the minimum and maximum component-wise values within the sample set, and the integrand by the same factor.
The scaled problem is equivalent to the unscaled version and the optimal $\bm{w}$ is the same.
To find the minimizing weights $\bm{w}$, we solve a standard-form quadratic program (QP) with linear constraints.
This is summarized in Routine~\ref{rout:QP} (see \cite{Amaral2017} for derivation details).
In Routine~\ref{rout:QP}, we use $\ell$ instead of $n$ for the number of samples because we consider two different algorithms in the following subsections that involve different numbers of samples. 
We also use $F_\text{targ}$ instead of the CDF $\obscdf$ since an EDF that approximates this CDF may also be used as discussed below.

\begin{routine}
\caption{$QP(\{\bm{q}^i,\ldots, \bm{q}^\ell\}\subset\mathbb{R}^d, F_{\text{targ}})$}
\begin{algorithmic}[1]
\label{rout:QP}
\STATE Set $H \in \mathbb{R}^{\ell \times \ell}$ and $\bm{b} \in \mathbb{R}^{\ell}$ as
\begin{equation*}
H_{ij} := \frac{1}{\ell^2}\prod^{d}_{k=1}\int_{z_k^{i,j}}^1 \, dq_k, \quad b_i := \frac{1}{\ell}\prod^d_{k=1}\int_{q_k^i}^1F_{\text{targ}}(\bm{q})\, dq_{k}
\end{equation*}
where $z_k^{i,j} = \max\{q_k^i, q_k^j\}$.
\STATE Solve for $\bm{w}=[w_1,\dots,w_\ell]^\top$:
\begin{equation}\label{eq:qp}
\begin{tabular}{c  c}
\text{minimize} & $\bm{w}^TH\bm{w} - b^T\bm{w}$\\
\text{subject to} & $\bm{w} \succeq \bm{0}$ \\
 & $\frac{1}{\ell}\sum^\ell_{i=1}w_i = 1$.
\end{tabular}
\end{equation}
\STATE Return $\bm{w}$.
\hfill $\triangleright$ Output
\end{algorithmic}
\end{routine}

The matrix $H$ constructed in the above routine is symmetric positive-definite so the QP has a unique solution \cite{cvx}. 
We emphasize that the matrix $H$ is dense and of order $\ell$, meaning that the computational complexity of the QP increases as the number of samples does.
However, the complexity of the QP is independent of the dimension of the problem.
In Step 2, any algorithm or package for solving QPs with linear constraints should suffice, but in this work we solve the QP using the convex optimization python package \textit{cvxopt} \cite{cvxopt}.
Running $QP(\set{\bm{q}^1,\ldots,\bm{q}^n},\obscdf)$ returns the vector of weights $\bm{w}\in\mathbb{R}^n$ such that $\predwedf$ defined in~\eqref{eq:predwedf} is the $L^2$-optimal estimate of $\obscdf$.
If the exact target CDF $F_{\text{targ}}$ is not known, as in the case when we have samples from the observed distribution instead of an exact form, we can apply the method using the EDF $F_{\text{targ}}^m$ for a set of $m$ iid samples from the observed distribution.
In that case, letting $\set{\bm{y}^1,\ldots,\bm{y}^m}$ denote the $m$ samples of the target distribution, the $b_i$ computation in Step 1 reduces to
\begin{equation*}
    b_i = \frac{1}{\ell}\prod^d_{k=1}\int_{q_k^i}^1F^m_{\text{targ}}(\bm{q})\, dq_{k} = \frac{1}{\ell}\prod^d_{k=1}\int_{q_k^i}^1 \frac{1}{m}\sum_{j=1}^m \mathbb{I}(\bm{y}^j \preceq \bm{q}) \, dq_k
\end{equation*}
It is shown in \cite{Amaral2017} that the resulting weighted EDF converges in the $L^1$- and $L^2$-norms to $F_{\text{targ}}$ as $\ell,m\to\infty$.
Applying Routine~\ref{rout:QP} to the illustrative example with $\ell=n=2E2$ samples (chosen to demonstrate the step-function nature of the solution), where the target distribution is an EDF of a normal distribution (using $m=1E4)$ results in the distribution approximation shown in Fig.~\ref{fig:data_space_compare}.
For the sake of comparison, we also show the results in the data space for the density-based method with $n=2E2$, i.e., we plot the densities for predicted, observed, and push-forward of the updated distribution. In both plots, we observe that the estimated predicted distribution is a poor estimate of the observed distribution whether it is represented as an EDF (left plot) or density (right plot).
The different re-weighting schemes both produce a change of measure resulting in an improved estimate of the observed distribution.
In the subsections below, we explore how to utilize the QP based re-weighting scheme to construct a weighted EDF approximation of the updated distribution. 
\begin{figure}
\centering
{\includegraphics[height=0.35\textwidth]{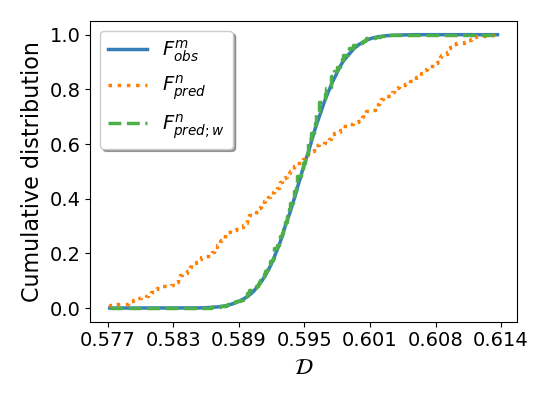}}
{\includegraphics[height=0.35\textwidth]{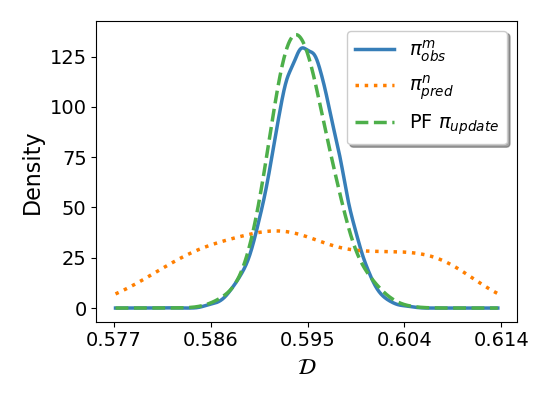}}
\caption{Data space comparisons of the optimal $L^2$ re-weighting scheme with the density-based approach for the illustrative example when $n=2E3$, $m=1E4$. On the left we plot the EDFs, on the right we plot the estimated densities.}
\label{fig:data_space_compare}
\end{figure}


\subsection{Na\"ive approach for utilizing an EDF}
\label{sec:cdfs_naive}
The first method we consider for solving the DCI problem is deemed the ``na\"ive'' approach.
For this approach, we use Routine~\ref{rout:QP} to compute a set of weights on the predicted samples (samples from the initial distribution that have been pushed forward through the map $Q$), and apply these weights directly on the corresponding initial samples.
This is summarized in Algorithm~\ref{alg:wEDFnaive}, which outputs $\initwedf$, defined in~\eqref{eq:wedfnaive}, as the solution.
This defines the following probability measure 
\begin{equation}
    \initwemeas(B) = \frac{1}{n}\sum_{i=1}^n w_i \mathbb{I}(\bm{\lambda}^i\in B), \quad B \in \mathcal{B}_\Lambda.
\end{equation}

\begin{algorithm}
\caption{Na\"ive approach}
\begin{algorithmic}[1]
\label{alg:wEDFnaive}
\STATE Set $n\in\mathbb{N}$, $\obscdf$, and $\initsamples$.
\hfill $\triangleright$ Input
\FOR {$1\leq i\leq n$}
    \STATE Evaluate $\bm{q}^i = Q(\bm{\lambda}^i)$
\ENDFOR
\STATE Use Routine \ref{rout:QP} to compute $\bm{w} = QP(\predsamples, \obscdf)$.
\STATE Apply $\bm{w}$ to initial samples $\initsamples$,
\begin{equation}
\label{eq:wedfnaive}
\initwedf(\bm{\lambda}) =\frac{1}{n}\sum^n_{i=1}w_i\mathbb{I}(\bm{\lambda} \preceq \bm{\lambda}^i).
\end{equation}
\STATE Return $\initwedf$.
\hfill $\triangleright$ Output
\end{algorithmic}
\end{algorithm}

\begin{figure}
\centering
{\includegraphics[height=0.35\textwidth]{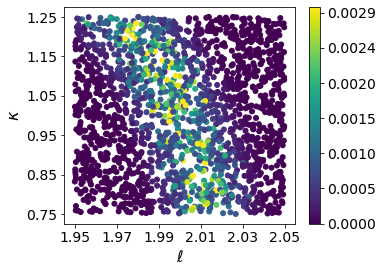}}
{\includegraphics[height=0.35\textwidth]{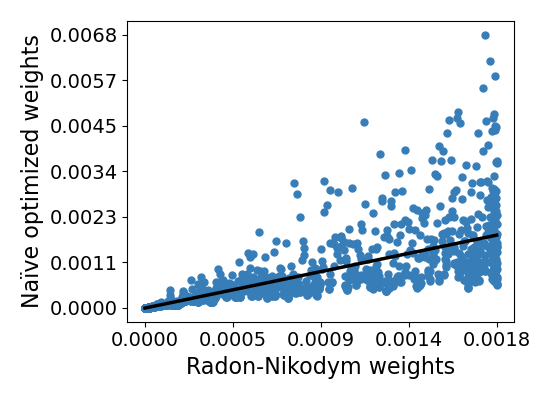}}
\caption{Weights resulting from na\"ive distributions-based method applied to initial samples on the left. Direct comparison to density-based weights on the right with the line indicating where perfect agreement occurs.}
\label{fig:naive_method}
\end{figure}

By comparing $r(\bm{\lambda})$ (shown in the left plot of Fig.~\ref{fig:dens_method}) to the weights, $\bm{w}$, computed from the na\"ive method (shown in the left plot of Fig.~\ref{fig:naive_method}), we observe that while the geometric structures of these weights are similar, the values are not the same as shown on the right side of Fig.~\ref{fig:naive_method}.
Not only does this na\"ive method place significantly more weight on certain samples, but the variation in the weights is also significantly greater than the variation in the density-based weights.
This is due to the simple fact that the QP is solved in the data space without any controls for variability in the parameter space.
Thus, while the solution is always optimal in the data space, it does not necessarily preserve any structure in the parameter space.
More precisely, this na\"ive approach does not require the solution to conform to the well-defined conditional structure dictated by the generalized contours of the map in the parameter space.
This observation motivates the binning approach in the next section.

\subsection{Binning approach}
\label{sec:cdfs_partition}
We now describe a two-step method to enforce the desired structure along the so-called ``contours'' defined by the pre-image map $Q^{-1}$ for which the conditionals defined by the disintegration of the initial probability measure are desired.
The key is to construct the ``binning distribution'' (defined below) associated with a particular partitioning of $\dspace$. Denote by $\{\mathcal{C}^k\}_{k=1}^p$ a partitioning of the data space.
For each cell $\mathcal{C}^k$, a representative point, denoted by $\bm{c}^k$, is identified.
These representative points could be the geometric center of the cell, the average of the predicted samples that fall into the cells, or any other suitable point in $\mathcal{C}^k$.
Alternatively, we may first generate the points and then consider the partition defined (implicitly) as Voronoi cells.
Each cell must be a \textit{continuity set} (i.e., a measurable set $A$ with $P(\partial A)=0$ where $\partial A$ denotes the topological boundary of $A$) of $\predmeas$, which in turn implies that it is a continuity set of $\obsmeas$, by Assumption \ref{ass:predictability}.

It is important to emphasize that the explicit construction of these cells is never actually required in either the theory or in practice, but we do reference the assumed continuity set property of these cells in the theoretical analysis of Section~\ref{sec:inv_theory}.
Further discussion on how to form an appropriate partition is found in Section \ref{sec:partitioning}.
For now, it suffices to understand this property as implying that the boundaries of such sets are both measurable and have zero-measure.

We solve a QP using the representative points as input samples in Algorithm~\ref{alg:wEDFnaive}, which effectively gives us an $L^2$-optimized weight for each (implicitly defined) \textit{cell}.
We define the resulting weighted EDF below. 
\begin{definition}[Binning Distribution]
The \textit{binning distribution} associated with $\{\bm{c}_k\}_{k=1}^p$ is defined as 
$F_{\dmeas;\bm{w}}^p = \frac{1}{p}\sum^p_{k=1}w_i\mathbb{I}(\bm{q} \preceq \bm{c}^k)$.
\end{definition}
Note here that because our sample set is now composed of $p$ samples, $\bm{w}$ is now a $p$-dimensional vector that is chosen separately from the $n$ parameter samples for which the QoI map is evaluated.
The subscript $\dmeas$ in $F_{\dmeas,\bm{w}}^p$ is utilized to emphasize that the partitioning utilized to construct the bins may be done a priori to any specification of a probability measure or generation of a random sample set on $\dspace$.
In other words, the measure of each bin may best be described, a priori, by utilizing $\dmeas$. 

To define a weighted EDF on $\pspace$, we utilize a classifier, denoted $Q^p$, to \textit{bin} the $n$ samples $\set{\bm{\lambda}^1,\ldots,\bm{\lambda}^n}$ into the $p$ (implicitly defined) sets $\set{Q^{-1}(\mathcal{C}^1),\dots,Q^{-1}(\mathcal{C}^p)}$.
It is at times conceptually convenient to also view the classifier $Q^p$ as defining a (discrete) map between the parameter and data space so that $Q^p(\bm{\lambda})=\bm{c}^k$ for a particular value of $1\leq k\leq p$. 
The meaning of the notation $Q^p$ as either a classifier or as a discrete map will always be clear from context.
This results in a vector of $n$ weights $\bm{u}=[u_1,\dots,u_n]^\top$, where $u_i$ corresponds to sample $\bm{\lambda}^i$, and is defined as $u_i = \frac{1}{n_k}w_k$ when $\bm{\lambda}^i \in Q^{-1}(\mathcal{C}^k)$ and $n_k$ is equal to the total number of parameter samples in $Q^{-1}(\mathcal{C}^k)$.
This method is described in detail below.

\begin{algorithm}
\caption{Binning approach}

\begin{algorithmic}[1]
\label{alg:wEDFbinning}
\STATE Set $p, n_\text{batch}\in\mathbb{N}$, $\initcdf$ and $\obscdf$
\hfill $\triangleright$ Input 
\STATE Partition $\mathcal{D} = \bigcup^p_{k=1}\mathcal{C}_k$
\STATE Identify representative point $c^k\in\mathcal{C}^k$ for each $k$
\STATE  Use Routine \ref{rout:QP} to compute $\bm{w} = QP(\{\bm{c}^k\}_{k=1}^p, \obscdf)$
\STATE Use $\bm{w}$ to set $\set{n_{k,\min}}_{k=1}^p$ with $n_{k,\min}\geq 0$ for all $k$
\STATE Set $n\gets 0$, $n_{k,\text{total}}\gets 0$ for $1\leq k\leq p$, initialize empty array of parameter samples $\mathcal{S}$, and define classifier $Q^p$ such that $Q^p(\lambda)= \bm{c}^k$ for some $1\leq k\leq p$ for all $\lambda\in\pspace$
\WHILE{$n_{k,\text{total}}<n_{k,\min}$ for any $k$}
     \STATE Generate $n_\text{batch}$ initial samples $\set{\bm{\lambda}^i}_{i=1}^{n_\text{batch}}$.
     \STATE Append $\set{\bm{\lambda}^i}_{i=1}^{n_\text{batch}}$ to $\mathcal{S}$, and set $n\leftarrow n+n_\text{batch}$ 
     \FOR {$1\leq i\leq n_\text{batch}$}
        \STATE Apply $Q^p$ to classify to which $\set{Q^{-1}(\mathcal{C}^k)}_{k=1}^p$ sample $\bm{\lambda}^i$ belongs
    \ENDFOR
     \FOR {$1\leq k\leq p$}
         \STATE $n_{k,\text{batch}}\gets $ number of $\set{\bm{\lambda}^i}_{i=1}^{n_\text{batch}}$ in $Q^{-1}(\mathcal{C}^k)$
         \STATE $n_{k,\text{total}}\gets n_{k,\text{total}} + n_{k,\text{batch}}$ 
     \ENDFOR
\ENDWHILE
\STATE Create array of $n$ scalar weights $\bm{u} = (u_1,\dots,u_n)$, as
    \begin{equation*}
    u_i = \begin{cases}
        \frac{w_k}{n_k}, & \text{if } n_{k,\min}>0, \\
        0, & \text{else}
    \end{cases}
    \end{equation*}
\STATE Construct weighted EDF,  $\initwedfpart$, on sample set $\mathcal{S}$ as
    \begin{equation}
    \label{eq:wedfbinning}
    \initwedfpart(\bm{\lambda}) = \sum^n_{i=1}u_i\mathbb{I}(\bm{\lambda} \preceq \bm{\lambda}^i).
    \end{equation}
\STATE Return $\initwedfpart$
\hfill $\triangleright$ Output
\end{algorithmic}
\end{algorithm}

Observe that the push-forward of $\initwedfpart$ through the discrete map $Q^p$ is, by construction, equal to $F_{\dmeas;\bm{w}}^p$ since we have
\begin{equation}
\label{eq:pf_def}
F_{\dmeas;\bm{w}}^p(\bm{q}) = \frac{1}{p}\sum^p_{k=1}w_k\mathbb{I}(\bm{q} \preceq \bm{c}^k) = \frac{1}{p}\sum^p_{k=1}\sum_{\{i : Q(\bm{\lambda^i}) \in \mathcal{C}^k\}}u_i\mathbb{I}(\bm{q} \preceq \bm{\lambda}^i) = \predwedfpart(\bm{q}).
\end{equation}
Thus, when applying this approach, we denote by $\predwedfpart$ the push-forward of $\initwedfpart$ via the map $Q^p$.
As $p$ and $n$ increase, $\predwedfpart$ and $\initwedfpart$ converge to the observed and update distribution, respectively, as shown in Section \ref{sec:theory}.

We illustrate the weights obtained from the binning method, where the partition is generated using a regular grid with $35$ bins, in the left plot of Fig.~\ref{fig:regpart_method}.
Comparing the right plots of Figs.~\ref{fig:naive_method} and~\ref{fig:regpart_method}, it is clear the resulting weights from the binning method are closer to the density-based weights, than the weights generated by the na\"ive method.
Note that the binning method assigns samples in the same bin the same weights, which explains the step-wise nature of the right plot in Fig.~\ref{fig:regpart_method}.

\begin{figure}
\centering
{\includegraphics[height=0.35\textwidth]{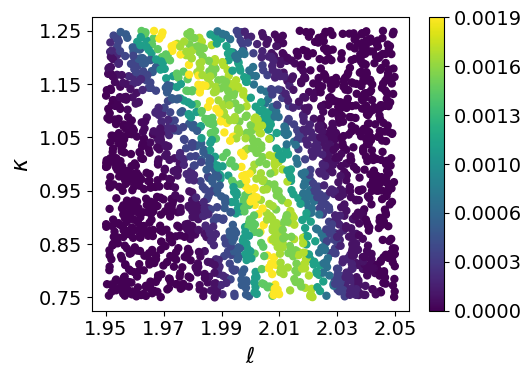}}
{\includegraphics[height=0.35\textwidth]{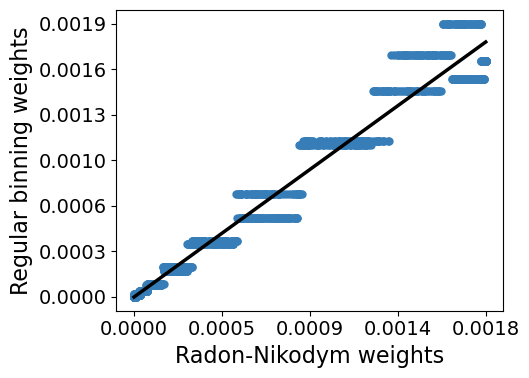}}
\caption{Weights resulting from the two-step partitioning method applied to initial samples on the left. Direct comparison to density-based weights on the right with the line indicating where perfect agreement occurs.}
\label{fig:regpart_method}
\end{figure}

\subsection{Construction of bins}
\label{sec:partitioning}
The binning method we have developed is similar to the two-step non-parametric importance sampling methods in \cite{Zhang1996} and \cite{Neddermeyer2009}.
There, a proposal distribution is approximated using a non-parametric density method, then classical importance sampling is performed.
Here, we are using a histogram-like non-parametric density estimation to approximate the predicted density on the data space before performing an optimization-based change of measure.

If we construct the cells using a regular grid partitioning scheme and representative points are chosen as the geometric centers of the cells, the first step is exactly a histogram density approximation.
When the support of the predicted density is irregular or the data space is very high dimensional, it may be more effective to use a data-driven partitioning scheme.
For example, we can cluster the predicted samples using K-means clustering \cite{arthur2006k}, and use variance-minimizing representative points.
For the convergence results of this paper to hold, we require a single additional assumption on the clustering method.
\begin{assumption}
\label{ass:clustering}
The binning distribution $F_{\mu_\mathcal{D}; \bm{w}}^{n,p}$ converges as $n,p \to \infty$ to a distribution that is absolutely continuous w.r.t. the observed distribution.
\end{assumption}

For the illustrative example, Fig.~\ref{fig:events} shows the binning of the samples in the parameter space that result from using both a regular grid partition and a K-means clustering method to create the cells on the data space.
The bins resulting from the K-means binning method still follow the contours of the mapping, as do the bins from the regular grid binning method, but they are more size-variable along the data-informed direction.

\begin{figure}
\centering
{\includegraphics[height=0.35\textwidth]{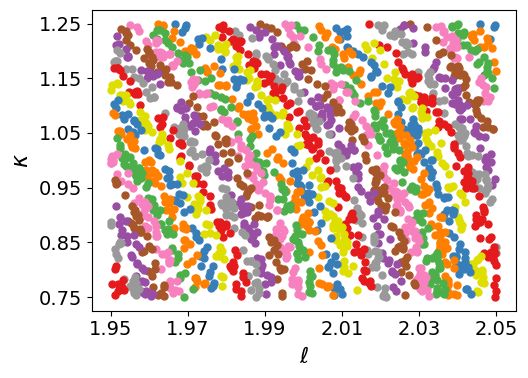}}
{\includegraphics[height=0.35\textwidth]{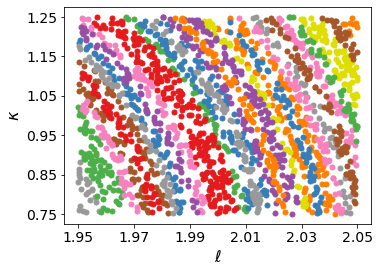}}
\caption{Binning of initial samples resulting from regular grid and K-means clustering partitioning methods using $40$ bins. Although the bins are computed on the output space, here we are showing the samples binned in the input space to illustrate how samples in the same bin fall in the same contour event.}
\label{fig:events}
\end{figure}

We plot the resulting weights obtained from the K-means clustering on the initial samples and compare to the density-based weights in the plots of Fig.~\ref{fig:kpart_method}.

\begin{figure}
\centering
{\includegraphics[height=0.35\textwidth]{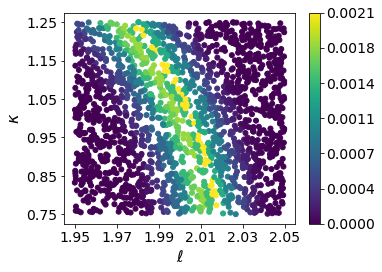}}
{\includegraphics[height=0.35\textwidth]{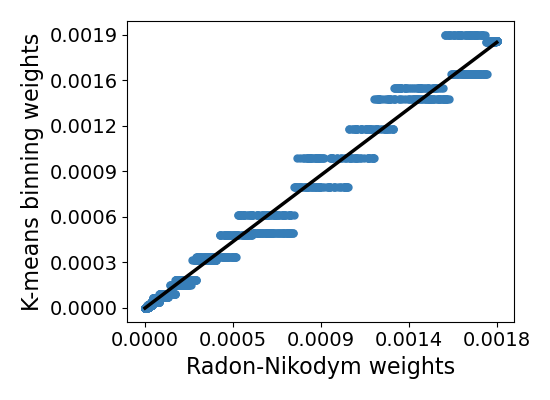}}
\caption{Weights resulting from the two-step partitioning method, where the partition is generated using a $k$-means classifier, applied to initial samples on the left. Direct comparison to density-based weights on the right.}
\label{fig:kpart_method}
\end{figure}

Visualizing the push-forward of the solution resulting from each method, as in Fig.~\ref{fig:part_res_compare}, we see that the binning method is not necessarily the $L^2$-optimal step function on the predicted samples.
For this figure, in order to more clearly see differences in the results for each of the methods, we re-ran the heat equation experiment with $m=10E4$ and $n=2E2$, and $p$, the number of bin, equal to $10$.
It may in fact result in a greater $L^2$-distance between the resulting weighted EDF in the data space and the observed than produced by the na\"ive method.
However, the result is still optimal for the weighted EDF over the partition centers, and as long as the partition is sufficient, it results in a relatively small loss in accuracy for the push-forward and remains data-consistent.
The solution itself is stable and accurate compared to the na\"ive method.

\begin{figure}
\centering
{\includegraphics[height=0.4\textwidth]{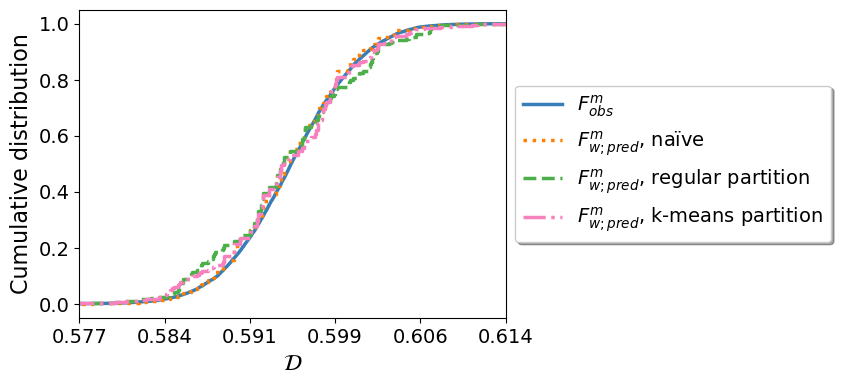}}
\caption{Push-forward distributions resulting from the various distribution-based methods. In both partition cases, we used $10$ bins. In all cases, $m=10E4$, $n=2E2$.}
\label{fig:part_res_compare}
\end{figure}
 
Next, we consider the case where at least one of the distributions involved does not admit a density.
In this case, the density based solution does not technically exist and application of the density method will result in clear inaccuracies.
For the heat equation example, we consider an observed distribution defined as a mixture of uniform distributions, $\obsmeas\sim 0.5\mathcal{U}((0.585,0.59]) + 0.1\mathcal{U}((0.59,0.595]) + 0.4\mathcal{U}((0.595,0.6])$,
resulting in a piecewise-linear distribution function.
We use $m=10E4$ observed samples and $n=2E4$ input/predicted samples.
The distributions involved and the result of the DCI density-based methods and binning-based methods are shown in Fig.~\ref{fig:no_dens}.
The right plot illustrates the errors the density-based method produces around the endpoints of the sub-intervals associated with each uniform distribution (where a ``density'' exhibits discontinuities) whereas the binning-based method produces a solution that is indistinguishable to the observed distribution.
\begin{figure}
\centering
{\includegraphics[height=0.33\textwidth]{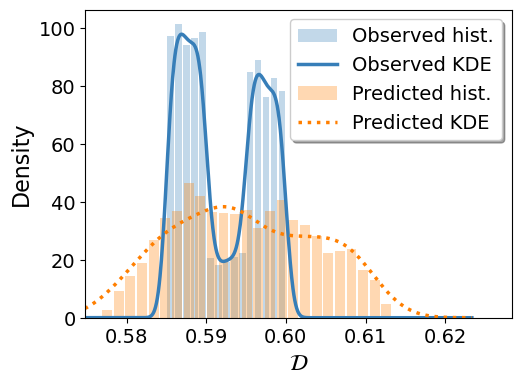}}
{\includegraphics[height=0.33\textwidth]{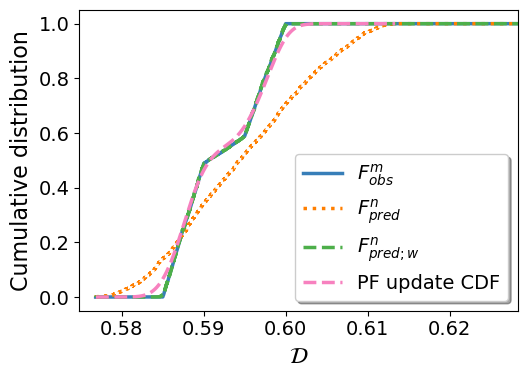}}
\caption{On the left: histograms of observed and predicted samples, and KDE estimations of the observed and predicted densities. On the right: DCI results for the binning and density-based methods. The CDF of the push-forward of the update computed from the density method (``PF update CDF'' in the plot) is clearly less accurate than the push-forward of the binning-based solution.}
\label{fig:no_dens}
\end{figure}

\section{Theoretical results}\label{sec:theory}

The goal of this section is to summarize existing results and develop the necessary theory required to prove that the weighted EDF obtained via Algorithm~\ref{alg:wEDFnaive}
converges to the data-consistent solution.
Section~\ref{sec:conv_dist} provides the theoretical results for multivariate distribution functions.
Specifically, we provide a theoretical connection between the $L^p$-convergence of multivariate distribution functions to the more commonly studied weak convergence.
For the \emph{univariate} case, $L^1$-convergence of \emph{univariate} distribution functions is equivalent to convergence in the Kantorovich/Wasserstein distance metric and implies weak convergence \cite{GibbsSu02}.
The weak convergence of multivariate distribution functions subsequently implies convergence on the so-called $P$-continuity sets, where $P$ is the probability measure associated with the limit of the distribution functions.
Recalling that the binning sets utilized in Algorithm~\ref{alg:wEDFnaive} are $\predmeas$-continuity sets, we then prove that the weighted EDFs converge to the data-consistent solution in Section \ref{sec:inv_theory}.

\subsection{Convergence of Multivariate Distribution Functions}
\label{sec:conv_dist}

We start with a formal definition of weak convergence for probability measures defined on a measurable space constructed from a metric space $S$ along with its Borel $\sigma$-algebra $\Sigma$.
\begin{definition}
A sequence of probability measures $\left\{P_n\right\}_{n\in\mathbb{N}}$ on $(S,\Sigma)$ converges weakly to probability measure $P$ if, for any $P$-continuity set $A$, $\lim_{n \to \infty} P_n(A) \to P(A)$.
\end{definition}

The following lemma relates this definition of weak convergence to the pointwise convergence of the corresponding multivariate distribution functions when $S=\mathbb{R}^d$.
\begin{lemma}
\label{lem:conv_dist}
Let $F$ be a distribution function on $\mathbb{R}^d$ corresponding to a probability measure $P$, and $\seq{F_n}$ a sequence of distribution functions corresponding to probability measures $\set{P_n}$.
If $F_n(\bm{x}) \to F(\bm{x})$ at all continuity points $\bm{x}\in\mathbb{R}^d$ of $F$, then $P_n\to P$ weakly.
\end{lemma}
\begin{proof}
The proof of this is provided within Example 2.3 of \cite{Billingsley1999}.
\end{proof}

The necessity of convergence at {\em every} continuity point of $F$ in the above lemma hints at a technical hurdle we must overcome.
Specifically, convergence of distribution functions in $L^p$ is not pointwise convergence. 
However, convergence in $L^p$ implies the existence of a subsequence $F_{n_k}\to F$ almost everywhere (a.e.).
It is not immediately obvious that all the continuity points of $F$ should belong to the a.e.~set nor is it immediately clear how to relate a subsequential limit to a limit of the original sequence of distribution functions.
We address these issues below.

The next lemma states that a multivariate CDF that converges a.e.~must also converge at all continuity points of the limit distribution.

\begin{lemma}\label{lem:cont_ptz}
Let $F$ be a multivariate distribution function on $\mathbb{R}^d$ and $\{F_n\}_{n\in\mathbb{N}}$ a sequence of approximate distribution functions. Suppose that $F_n \to F$ a.e., then $F_n \to F$ at all continuity points of $F$.
\end{lemma}

\begin{proof}
    See Appendix \ref{app:cont_ptz}.
\end{proof}

We now state the main result of this subsection that connects $L^p$ convergence of multivariate distribution functions to weak convergence of probability measures..

\begin{theorem}\label{thm:am_ext}
Let $F$ be a multivariate distribution function on $\mathbb{R}^d$ corresponding to a probability measure $P$, and $\{F_n\}_{n\in\mathbb{N}}$ a sequence of distribution functions corresponding to probability measures $\{P_n\}_{n\in\mathbb{N}}$. 
If $F_n \to F$ in $L^p$, where $p \ge 1$, then $P_n \to P$ weakly.
\end{theorem}

\begin{proof}
    See Appendix \ref{app:am_ext}.
\end{proof}

Theorem~\ref{thm:am_ext} immediately extends a result given in~\cite{Amaral2017} involving the convergence of the $L^2$-based optimization method for re-weighting the empirical distribution function in the data space that we exploit within Algorithm~\ref{alg:wEDFbinning}.
Specifically, between Theorem 1 and Corollary 1 in~\cite{Amaral2017}, the authors demonstrate that the optimal $L^2$-convergence of distribution functions implies $L^1$-convergence.
It is then commented that in the univariate case (i.e., for distribution functions defined on $\mathbb{R}$) this is equivalent to convergence in the Kantorovich/Wasserstein distance from which weak convergence follows under the additional assumption of bounded support of the proposal measure. 
By referring to Theorem~\ref{thm:am_ext} instead of the Kantorovich/Wasserstein distance, it immediately follows that the optimization approach developed in~\cite{Amaral2017}, and utilized in the QP portion of Algorithm~\ref{alg:wEDFbinning}, produces a sequence of distribution functions and corresponding probability measures that weakly converge in $\mathbb{R}^d$ for any $d\geq 1$.
We summarize this observation as a corollary.
\begin{corollary}
Let $F$ be a distribution function on $\mathbb{R}^d$ corresponding to a probability measure $P$, and $\{F^n_{\bm{w}}\}$ a sequence of distribution functions resulting from applying Algorithm \ref{alg:wEDFnaive} to a set of $n$ samples $\{\bm{x}^1,\dots,\bm{x}^n\}$ from a distribution that is absolutely continuous with respect to $P$.
Then, the probability distributions corresponding to the distribution functions $\seq{F_{{\bm{w}}}^n}$ converge weakly to $F$ as $n\to\infty$.
\end{corollary}

\subsection{Convergence of EDF Approximations of Solutions to the Inverse Problem}
\label{sec:inv_theory}

Returning to the inverse problem, the goal is to use Theorem~\ref{thm:am_ext} to show that under certain conditions the binning-based solution obtained from Algorithm~\ref{alg:wEDFbinning} produces EDFs such that the associated probabilities of certain events on $\pspace$ converge to the probabilities given by the data-consistent solution $\upmeas$.

Recall from~\eqref{eq:pf_def} that the push-forward of $\initwedfpart$ through the map $Q^p$ is, by construction, equal to $F_{\dmeas;\bm{w}}^p = \frac{1}{p}\sum^p_{k=1}w_i\mathbb{I}(\bm{q} \preceq \bm{c}^k)$.
Denoting by $\predwedfpart$ the push-forward of $\initwedfpart$ via the map $Q^p$ proves the following Lemma.

\begin{lemma}\label{lem:optweights_converge}
Let $\set{\mathcal{C}_k}_{k=1}^{p}$ be a $\obsmeas$-continuity partition of $\dspace$, with associated weights $\bm{w} \in \mathbb{R}^p$ generated by Algorithm \ref{alg:wEDFbinning}, and let $\set{\bm{\lambda}^j}_{j=1}^n$ be the corresponding sequence of random samples in $\Lambda$ with computed weights $\bm{u}\in\mathbb{R}^n$ from Algorithm~\ref{alg:wEDFbinning}.
Together, these weights and samples define a predicted distribution $\predwemeaspart$ such that for every $\obsmeas$-continuity set $D$, 
\begin{equation*}
    \lim_{p\to\infty}\lim_{n\to\infty}\predwemeaspart(D) = \obsmeas(D).
\end{equation*}
In other words, $\predwemeaspart \to \obsmeas$ weakly as $p, n\to\infty$.
\end{lemma}

We now formally state the main result of this paper.

\begin{theorem}\label{thm:dci_binning}
Assume the predictability assumption holds and let $\upcdf$ and $\upmeas$ denote the distribution and probability measure, respectively, defined by the density-based approach. 
Under the conditions of Lemma~\ref{lem:optweights_converge}, the sequence of distribution functions $\set{\initwedfpart}_{n,p=1}^\infty$ and associated probability measures $\set{\initwemeaspart}_{n,p=1}^\infty$ have the following properties.
\begin{enumerate}[(i)]
    \item Defining the push-forward $F_{pred;\bm{u}}^{n,p}$ as in Lemma~\ref{lem:optweights_converge}, $\predwedfpart\to\obscdf$ weakly as $n,p\to\infty$.
    \item Let $A\in\pborel$ with $Q(A)\in\dborel$ a continuity set of $\obsmeas$ and $\obsmeas(Q(A))>0$. If $\set{\bm{\lambda}^j}_{j=1}^n$ is an iid set drawn from $\initmeas$, then $\initwemeaspart(A)\to \upmeas(A)$ as $n,p\to \infty$.
\end{enumerate}
\end{theorem}

\begin{proof}
    See Appendix \ref{app:dci_binning}.
\end{proof}

Part (i) of the above theorem states that the sequence of pullback distributions constructed via Algorithm~\ref{alg:wEDFbinning} push forward to distributions that converge weakly to the observed distribution.
Part (ii) of the above theorem describes the convergence on the parameter space of these pullback distributions.
The condition that $\obsmeas(Q(A))>0$ is used to avoid conditioning on sets of zero measure.

\section{Numerical results}
\label{sec:apps}

\subsection{Convergence}
\label{sec:convergence}
To demonstrate convergence in both the data space $\mathcal{D}$ and the parameter space $\Lambda$, we return to the example introduced in Section~\ref{sec:heat_eq_ex}.
We begin with a reference set in $\Lambda$, push this set through the map $Q$, and then consider the corresponding contour set.
To this end, we take $A = [2.01,2.02] \times [0.95,1.0]\subset \Lambda$ to define $B = Q(A) \approx [0.59,0.5936]\subset\dspace$, and $Q^{-1}(B)\subset\pspace$ (see Fig.~\ref{fig:sets}).

Baseline estimates of probabilities for these sets are generated with the density-based DCI method and repeated trials utilizing distinct initial samples for each trial (the observed samples are not regenerated) to compute the average probabilities for a more accurate approximation.
We set $m=1E5$, $n = 1E5$, and utilize 10 trials to obtain $\obsmeas(B) \approx 0.26697$ and $\upmeas(A) \approx 0.01986$.

\begin{figure}
\centering
{\includegraphics[height=0.335\textwidth]{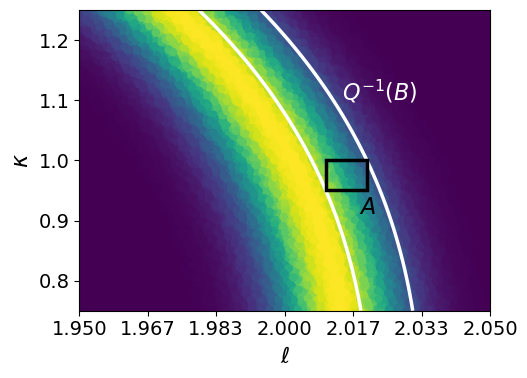}}
{\includegraphics[height=0.365\textwidth]{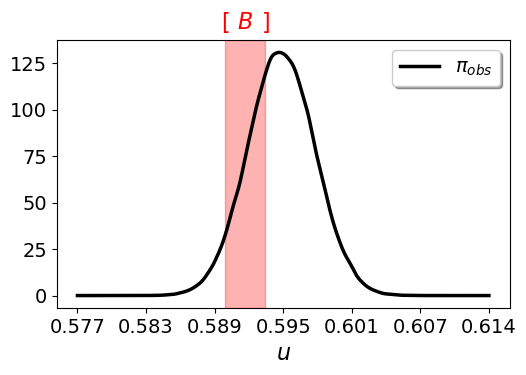}}
\caption{Sets $A$ and $Q^{-1}(B)$ in $\Lambda$ on the left. Set $B$ on the right.}
\label{fig:sets}
\end{figure}

We now utilize the set $B$ to illustrate the weak convergence for probabilities in $\mathcal{D}$ as $n,p \to \infty$ (see Theorem \ref{thm:dci_binning} (i)).
Sets of $n$ (ranging from $1E3$ to $1E4$) initial samples are drawn.
Following Algorithm \ref{alg:wEDFbinning}, each set of initial samples is drawn by \textit{appending} samples to the previous set.
Once the initial samples are drawn, we create $p$ bins (ranging from $20$ to $160$) using regular linear spacing on $\mathcal{D}$.
We then compute the optimization-based weights and estimate the probability of $B$ as the sum of weights on samples binned in $B$.
Figs.~\ref{fig:meanB} and~\ref{fig:stdB} summarize the mean error and standard deviation of these results taken over 20 trials.
Trends are observed as we move from the upper left corner (corresponding to small $n$ and $p$ values) to the lower right corner (corresponding to larger $n$ and $p$ values) in both figures.
Specifically, Fig.~\ref{fig:meanB} illustrates a trend towards more accurate approximations while Fig.~\ref{fig:stdB} illustrates a trend towards reduced variance in these estimates. 
Note that the set $A$ is quite small in comparison to the size of the parameter space, so that estimate variance is primarily driven by sample size $n$ instead of bin number $p$.
Consider, for example, that in one trial with $n=1E4$, only $89$ samples fall in $A$.

\begin{figure}
\centering
{\includegraphics[height=0.4\textwidth]{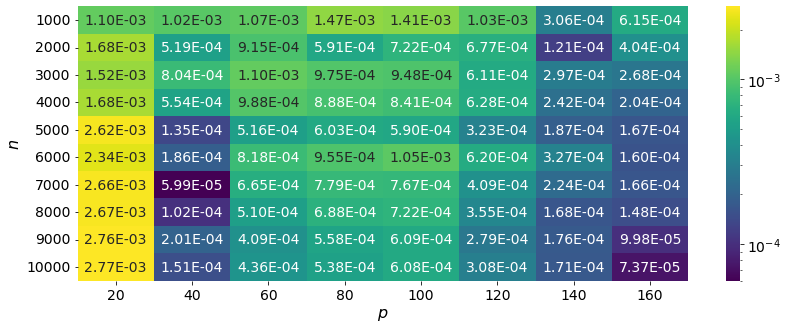}}
\caption{Absolute error between mean of $\predwemeaspart(B)$ over 20 trials and $\obsmeas(B)$.}
\label{fig:meanB}
\end{figure}

\begin{figure}
\centering
{\includegraphics[height=0.4\textwidth]{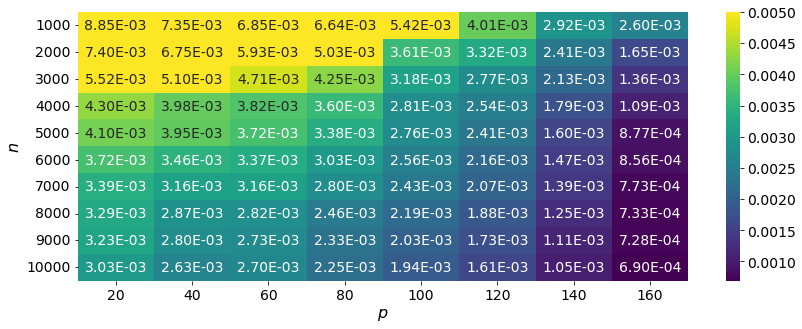}}
\caption{Standard deviation over 20 trials of $\predwemeaspart(B)$.}
\label{fig:stdB}
\end{figure}

We now illustrate the weak convergence for probability in $\pspace$ in Theorem~\ref{thm:dci_binning} (ii) with the set $A$.
Fig.~\ref{fig:meanA} shows  that as $n,p \to \infty$ the absolute error decreases.
The standard deviation of the results also decreases as $n,p\to\infty$ as shown in Fig.~\ref{fig:stdA}.

\begin{figure}
\centering
{\includegraphics[height=0.4\textwidth]{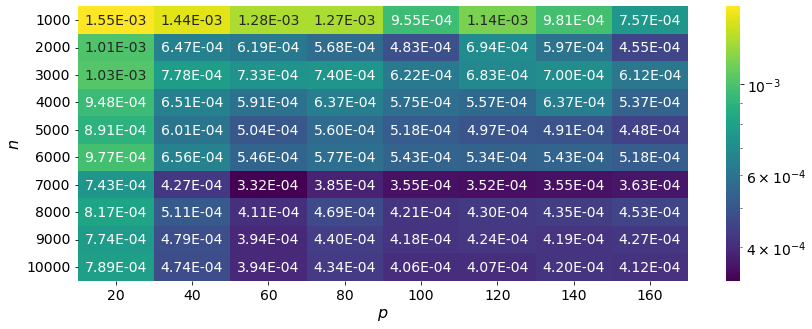}}
\caption{Absolute error between mean of $\initwemeaspart(A)$ over 100 trials and $\obsmeas(A)$.}
\label{fig:meanA}
\end{figure}

\begin{figure}
\centering
{\includegraphics[height=0.4\textwidth]{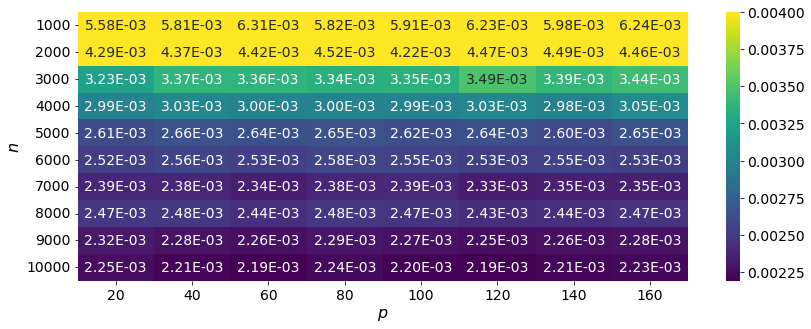}}
\caption{Standard deviation over 20 trials of $\predwemeaspart(A)$.}
\label{fig:stdA}
\end{figure}

\subsection{Fluid flow through porous media}
\label{sec:fluid_flow}
We now consider a 3D porous media model based on the SPE10 data set~\cite{SPE10}. 
The physical domain is a 3-dimensional slab: $\Omega=[0,1200]\times[0,2200]\times[0,100]$, where the coordinates are given in feet. 
The model for single phase incompressible flow is given by,
\begin{equation}\label{eq:porous}
\begin{cases}
 \mathbf{u} = -\mathbf{K} \nabla p, & x\in\Omega, \\
 \nabla \cdot \mathbf{u} = 0, & x\in\partial\Omega,
 \end{cases}
 \end{equation}
 where $\mathbf{K}$ is a heterogeneous permeability field.
 We apply a pressure drop from one side (high $y$ values) to the other (low $y$ values) by fixing the pressure along these faces and applying no-flow boundary conditions ($\mathbf{u} \cdot \mathbf{n} = 0)$ on the remaining faces. The mesh is a uniform grid of $60\times220\times50$ elements, and we use a hybridized mixed formulation~\cite{fortin1991mixed}, which results in 6.6 million degrees of freedom, where 2.2 million are global/hybrid
and 4.4 million are local/subgrid. 

The input space for the inverse problem is defined by the $3$-dimensional starting positions for a set of streamlines, generated using ParaView~\cite{ParaView}, which flow through the slab (generally in the decreasing $y$-direction).
The output space is defined by the $3$-dimensional final positions (we simplify the output space to $1$-dimension as described below) for the streamlines.
These final positions are given when the flow stops or the streamlines reach $y=0$ (i.e., when they exit the slab).
The pressure field, the magnitude of the velocity field, and the set of streamlines associated with the samples from the initial density are shown in Fig.~\ref{fig:fluid_flow_solns}.
\begin{figure}
\centering
\raisebox{-0.5\height}{\includegraphics[height=0.4\textwidth]{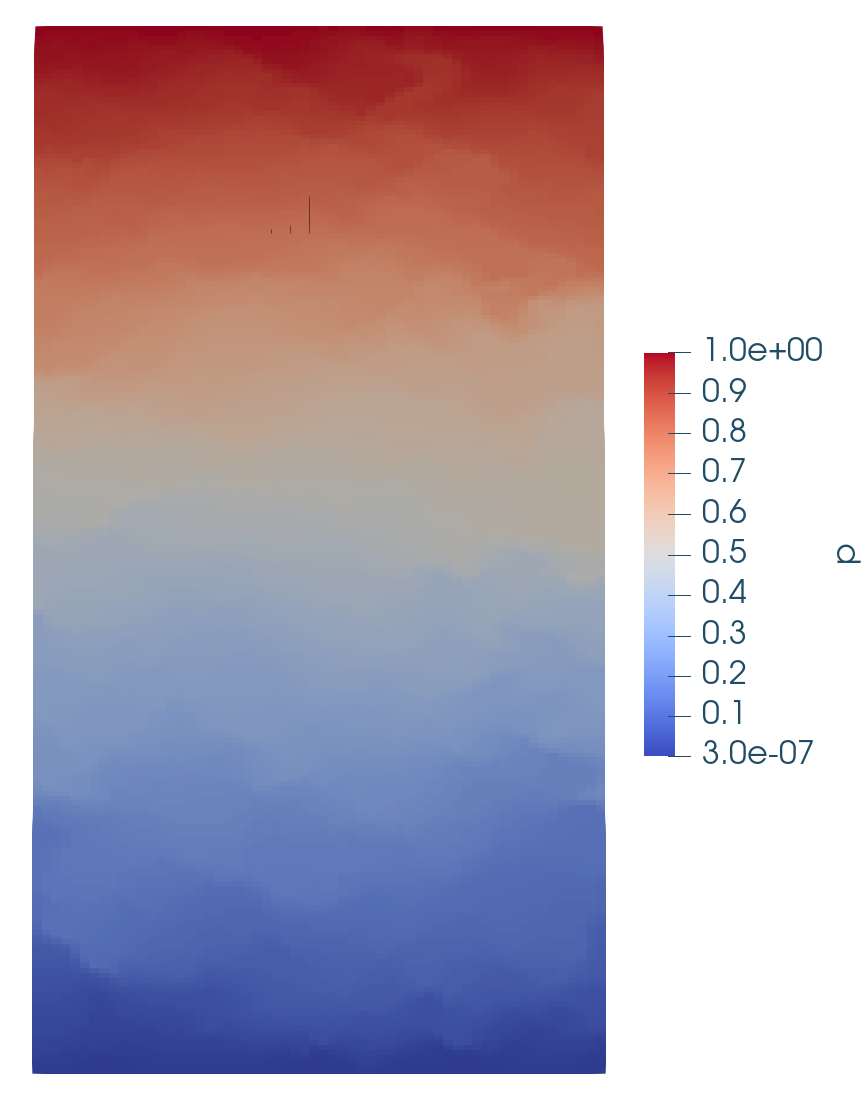}}
\raisebox{-0.5\height}{\includegraphics[height=0.4\textwidth]{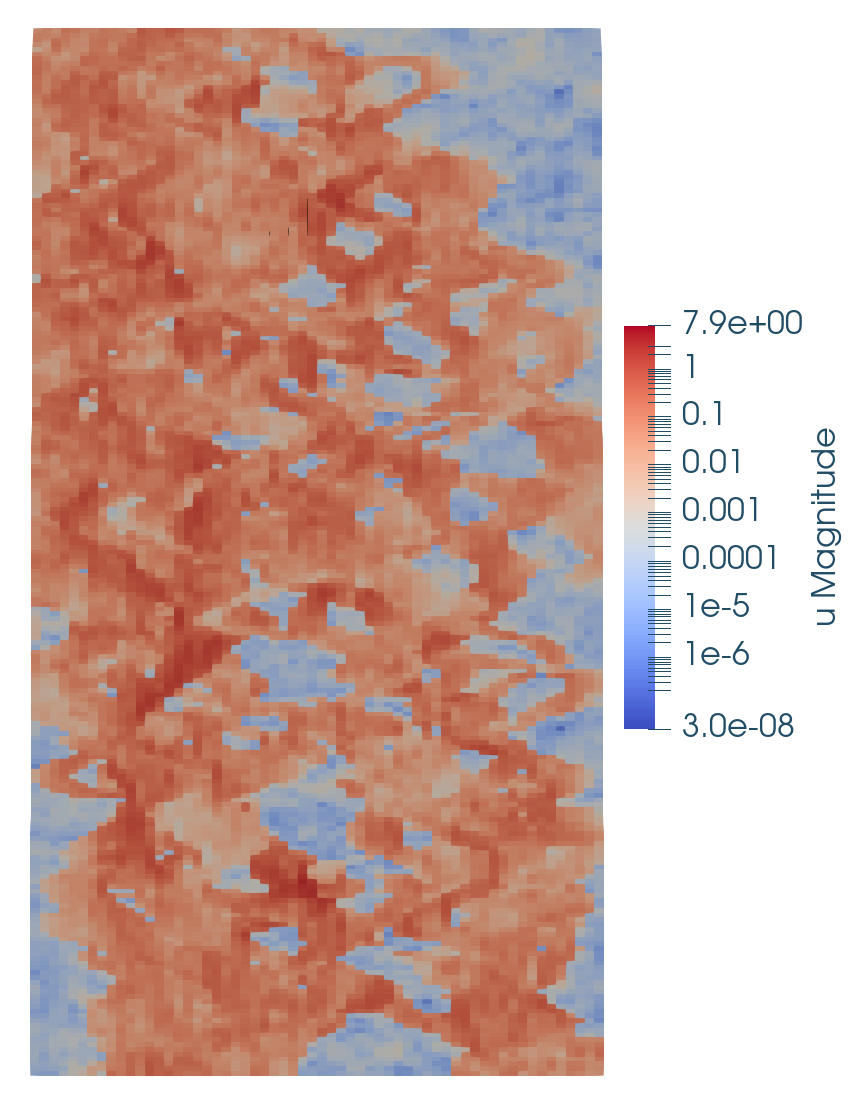}}
\raisebox{-0.5\height}{\includegraphics[height=0.4\textwidth]{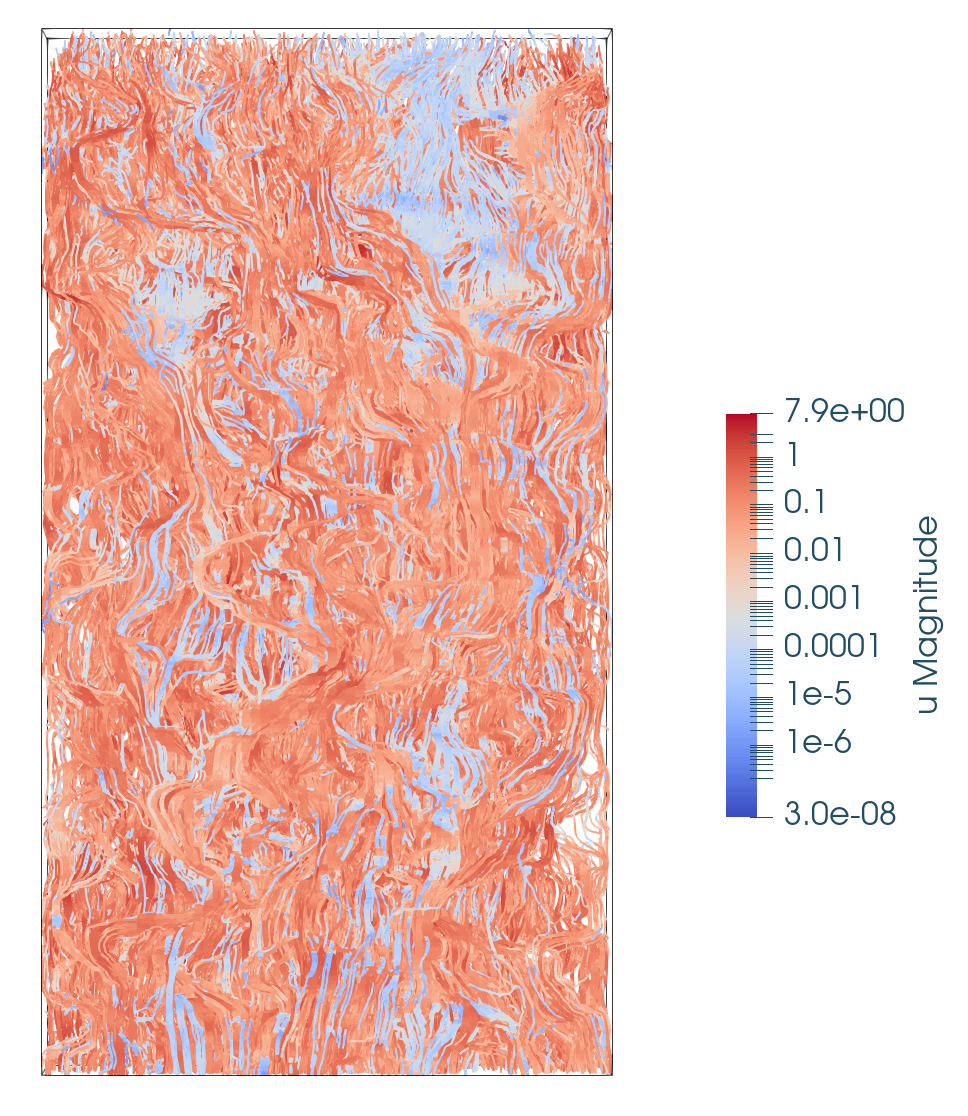}}
\caption{On the left: the numerical approximation of the pressure field.
In the middle: the numerical approximation of the magnitude of the velocity field.
On the right: The streamlines associated with the samples from the initial distribution.}
\label{fig:fluid_flow_solns}
\end{figure}
The starting and ending points of the initial samples are show on the left of Fig.~\ref{fig:fluid_flow_samples}. 
They are generated uniformly in the intersection of a $3$-dimensional sphere and the slab. 
We are interested in the initial samples that either exit or get close to exiting the slab by the end of the simulation; that is, samples for which the $y$ value of their finish positions are less than $100$. 
We then focus on the $x$ position of the streamlines,
resulting in a $1$-dimensional output space. 
The observed distribution is $\mathcal{N}(900,25)$ and we generate $1E4$ points from the observed distribution to serve as observed samples.
We illustrate each output distribution using a KDE in the right of Fig.~\ref{fig:fluid_flow_samples}.

\begin{figure}
\centering
\raisebox{-0.5\height}{\includegraphics[height=0.45\textwidth]{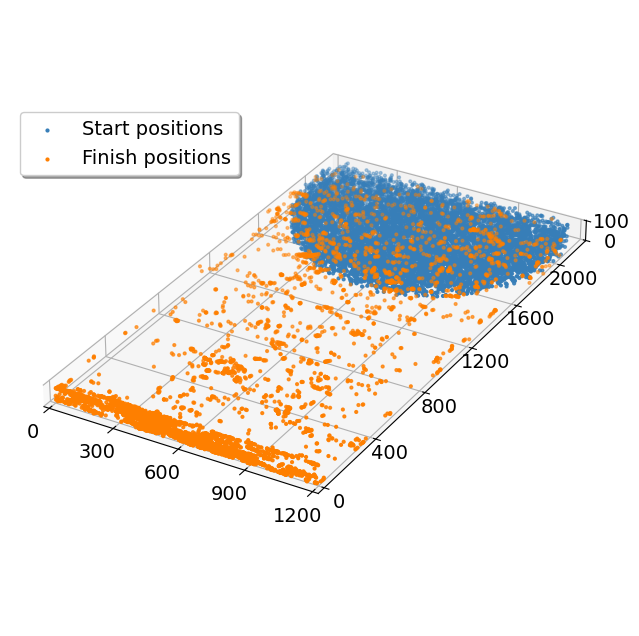}}
\raisebox{-0.5\height}{\includegraphics[height=0.35\textwidth]{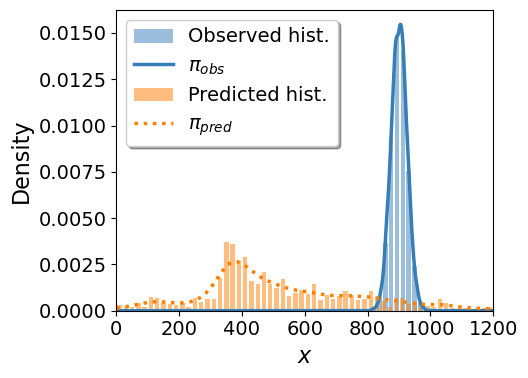}}
\caption{On the left: starting positions for the initial samples shown in blue, and the end points for these samples shown in orange. On the right: the predicted samples and KDE shown in orange, and the observed samples and KDE shown in blue.}
\label{fig:fluid_flow_samples}
\end{figure}

We apply the density method along with rejection sampling to find a set of samples from the updated density. 
The mean $r$-value of the density method is approximately $1.08$, indicating good, but not excellent, results. 
We show the accepted samples in the input space, as well as the push-forward of these samples in the data space, in Fig.~\ref{fig:fluid_flow_density_res}. 
We observe that the update samples are clustered towards a region with higher $x$-values.
This makes sense considering the spatial properties of the problem. 
In the data space, we see that the results are also reasonable, although the push-forward of the update solution does not quite match the observed.

\begin{figure}
\centering
\raisebox{-0.5\height}{\includegraphics[height=0.45\textwidth]{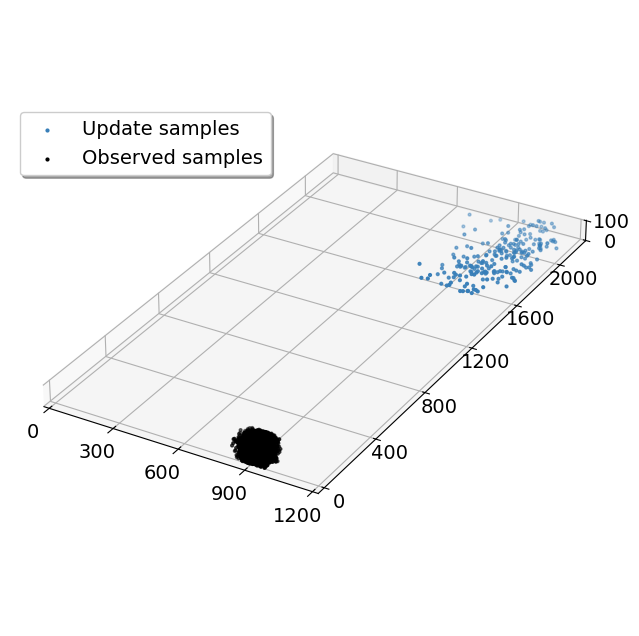}}
\raisebox{-0.5\height}{\includegraphics[height=0.35\textwidth]{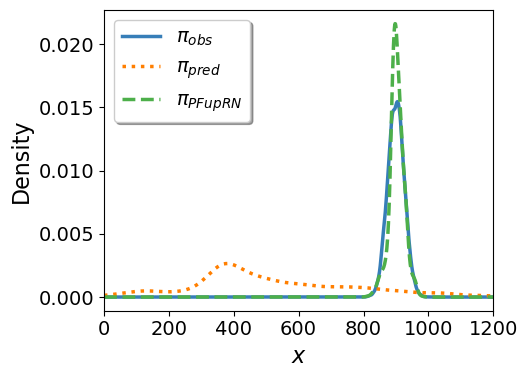}}
\caption{On the left: update samples shown in blue, observed shown in black. Because the observed samples are $1$-dimensional, we plot them with $y$ and $z$ both uniformly and randomly distributed from $0$ to $100$. On the right: push-forward results in the data space from the density method.}
\label{fig:fluid_flow_density_res}
\end{figure}

Next, we apply the K-means binning method. 
Because the majority of the probability for the observed distribution is centered in a small region of the support of the predicted, we require a relatively large number of bins, which informs our choice of $p=120$ bins. 
The results, shown as weighted EDFs on the data space, as well as the weights plotted on the initial samples, are given in Fig.~\ref{fig:fluid_flow_kmeans_res}.
We observe that the binning approach produces larger weights for the samples that are mostly clustered towards a region of  higher $x$-values, which is both what we expect and agrees with the results from the density solution. 
Unlike the density-based method that produced some clearly identifiable approximation errors in the push-forward of the updated density, the resulting CDFs show that the push-forward of the solution from the binning method appears to be identical to the observed distribution on the data space. 

\begin{figure}
\centering
\raisebox{-0.5\height}{\includegraphics[height=0.43\textwidth]{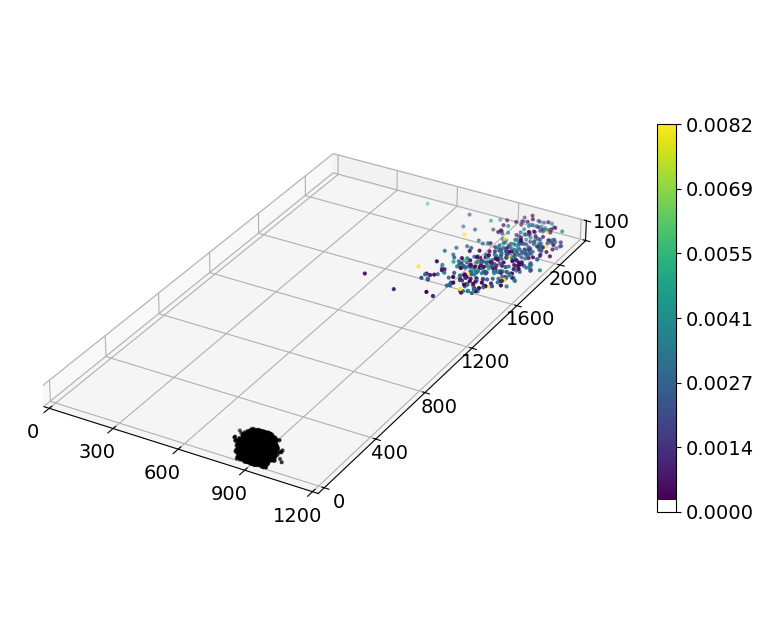}}
\raisebox{-0.5\height}{\includegraphics[height=0.32\textwidth]{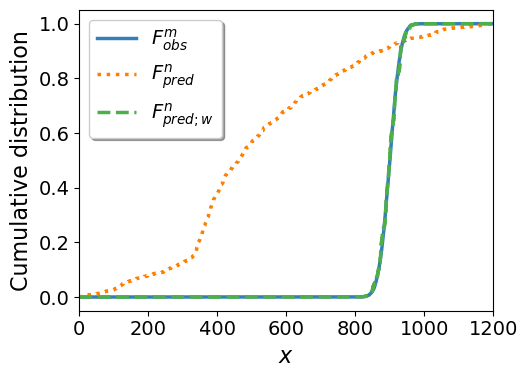}}
\caption{On the left: solution weights from the kmeans binning method shown in color, observed shown in black. For readability, and because the majority of the weights are very close to zero, we remove all samples where the weights are less than $1E-4$ from the plot. Because the observed samples are $1$-dimensional, we plot them with $y$ and $z$ both uniformly and randomly distributed from $0$ to $100$. On the right: push-forward results in the data space from the kmeans binning method.}
\label{fig:fluid_flow_kmeans_res}
\end{figure}

\section{Conclusions}
\label{sec:conclusions}

In this work, we develop a distributions-based approach for solving data-consistent stochastic inverse problems. 
While previous approaches to solving such problems require approximations of events or densities, this work provides a rigorous methodology that applies to cases with limited observational or simulation/prediction data as well as situations where densities and events cannot be approximated effectively.
This builds upon prior work by incorporating a binning approach in the data space to properly distribute probabilistic weights in the input space.
Proofs of convergence of the distributions-based solution to the density-based solution as the number of samples and bins increase are provided.

In future work, we aim to formalize a theoretical proof of the stability of the pullback measure that is defined by the weighted EDF on the parameter space, as well as a theoretical proof to justify the generalization of the method to more arbitrary input sample sets (i.e., not just those generated according to a specified initial distribution). 
We are also interested in quantifying the effect of the model form errors on the weighted EDFs and the corresponding approximations to the pullback measure. 
Future directions will also include a utilization of this approach in the context of optimal experimental design and an exploration of the use of this approach to efficiently perform hierarchical Bayesian inference.

\section{Code}
\label{sec:code}
Code to reproduce the figures and examples in this paper is contained in the GitHub repository~\cite{distDCI} for this paper.

\section{Acknowledgments}
\label{sec:ack}
T.~Butler's work is supported by the National Science Foundation under Grant No.~DMS-2208460 as well as by the NSF IR/D program, while working at National Science Foundation. 
However, any opinion, finding, and conclusions or recommendations expressed in this material are those of the author and do not necessarily reflect the views of the National Science Foundation.
T.~Wildey's work is supported by the Office of Science, Early Career Research Program.
This paper describes objective technical results and analysis. Any subjective views or opinions that might be expressed in the paper do not necessarily represent the views of the U.S. Department of Energy or the United States Government.

\bibliographystyle{siamplain}
\bibliography{references}

\appendix

\section{Proof of Lemma \ref{lem:cont_ptz}}
\label{app:cont_ptz}

Before beginning the proof, we recall again the notions of ``greater than'' and ``less than'' in multivariate spaces.
Specifically, we use the standard notation $\succeq$ and $\preceq$ to indicate component-wise greater than and component-wise less than, respectively; i.e., $\bm{x} =(x_1,\ldots, x_d)\succeq \bm{y}=(y_1,\ldots,y_d)$ if and only if $x_1 > y_1, \dots, x_d > y_d$.
If a function is ``nondecreasing'' we mean this in the component-wise sense: $F(\bm{x}) \ge F(\bm{y})$ if $\bm{x} \succeq \bm{y}$.

\begin{proof}
Let $\bm{x}$ be a continuity point of $F$. 
Let $A = \{\bm{z} : F_n(\bm{z}) \to F(\bm{z})\}$, which we refer to as an ``a.e.~set'' meaning the Lebesgue measure of the complement of $A$ is zero. 
For $\delta>0$, define $\bm{x}-\delta$ as the point $(x_1-\delta,\ldots,x_d-\delta)$.
Then, $\{\bm{z} : \bm{x}-\delta \preceq \bm{z} \preceq \bm{x}\}$ defines a $d$-dimensional cube of Lebesgue measure $\delta^d>0$.
Since $A$ is an a.e.~set, there exists an (uncountably) infinite number of points in $A\cap \{\bm{z} : \bm{x}-\delta \preceq \bm{z} \preceq \bm{x}\}$.
It is therefore possible to choose a  sequence $\{\bm{l}_i\}_{i=1} \subset A \cap \{\bm{z} : \bm{z} \preceq \bm{x}\}$
such that $\bm{l}_i \to \bm{x}$.
An analogous argument implies we can also choose a sequence $\{\bm{u}_i\}_{i=1} \subset A \cap \{\bm{z} : \bm{z} \succeq \bm{x}\}$ such that $\bm{u}_i \to \bm{x}$.
Because each $F_n$ is nondecreasing in the component-wise sense, for any $n \in \mathbb{N}$ and any $i \in \mathbb{N}$, we have
\begin{equation*}
F_n(\bm{l}_i) \le F_n(\bm{x}) \le F_n(\bm{u}_i).
\end{equation*}
The goal is to prove $F_n(\bm{x})\to F(\bm{x})$.
To do this, we first observe that
\begin{equation*}
    \liminf_{n\to\infty} F_n(\bm{l}_i) \le \liminf_{n\to\infty} F_n(\bm{x}) \le \liminf_{n\to\infty} F_n(\bm{u}_i)
\end{equation*}
and
\begin{equation*}
    \limsup_{n\to\infty} F_n(\bm{l}_i) \le \limsup_{n\to\infty} F_n(\bm{x}) \le \limsup_{n\to\infty} F_n(\bm{u}_i).
\end{equation*}

Now, as $n \to \infty$, $F_n(\bm{l}_i) \to F(\bm{l}_i)$ because $\bm{l}_i \in A$ for each $i$, and similarly $F_n(\bm{u}_i) \to F(\bm{u}_i)$ for each $i$. Thus, 
\begin{equation*}
F(\bm{l}_i) = \lim_{n\to\infty}F_n(\bm{l}_i) \le \liminf_{n\to\infty}F_n(\bm{x}) 
\end{equation*}
and
\begin{equation*}
 \limsup_{n\to\infty}F_n(\bm{x}) \le \lim_{n\to\infty}F_n(\bm{u}_i) = F(\bm{u}_i).
\end{equation*}
As $i \to \infty$, by construction we have $\bm{l}_i \to \bm{x}$ and $\bm{u}_i \to \bm{x}$, and $\bm{x}$ is a continuity point of $F$, so we have
\begin{equation*}
\lim_{i\to\infty} F(\bm{l}_i) = F(\bm{x}) \le \lim_{i \to \infty}\left(\liminf_{n\to\infty}F_n(\bm{x})\right) = \liminf_{n\to\infty}F_n(\bm{x})
\end{equation*}
and
\begin{equation*}
\lim_{i \to \infty}\left(\limsup_{n\to\infty}F_n(\bm{x})\right) = \limsup_{n\to\infty}F_n(\bm{x}) \le F(\bm{x}) = \lim_{i\to\infty} F(\bm{u}_i).
\end{equation*}
Thus, we have found
\begin{equation*}
F(\bm{x}) \le \liminf_{n\to\infty}F_n(\bm{x})  \le \limsup_{n\to\infty}F_n(\bm{x}) \le F(\bm{x})
\end{equation*}
which implies that $\lim_{n\to\infty} F_x(\bm{x})$ exists and is equal to $F(\bm{x})$. Since $\bm{x}$ was an arbitrary continuity point of $F$, the conclusion follows.
\end{proof}

\section{Proof of Theorem \ref{thm:am_ext}}
\label{app:am_ext}

Before we prove this theorem, we recall a useful result connecting the convergence of subsequences of subsequences to the original sequence that is straightforward to prove but can seem confusing when used without context. 

\begin{lemma}\label{lem:subsub}
Let $\{p_n\}_{n\in\mathbb{N}}$ be a sequence of real numbers.
If there exists $p \in \mathbb{R}$ such that for any subsequence $\{p_{n_k}\}_{k\in\mathbb{N}}$ there exists a further subsequence $\{p_{n_{k_l}}\}_{l \in\mathbb{N}}$ that converges to $p$, then $p_n \to p$.
\end{lemma}

While Lemma~\ref{lem:subsub} is a standard result in analysis, the proof can be difficult to find as it is often left as an exercise.
We provide the proof below for ease of reference.
\begin{proof}
Suppose that $(p_n)$ does not converge to $p$.
Then there exists a $\varepsilon > 0$ such that for every $N \in \mathbb{N}$ there exists $n \ge N$ such that $\abs{p_n - p} \ge \varepsilon$.
Choose such an $\varepsilon > 0$, and inductively construct a subsequence $(p_{n_k})$ as follows. For $k = 1$, $n_1 \ge 1$ such that $\abs{p_{n_1} - p} \ge \varepsilon$, for $k = 2$, choose $n_2 \ge n_1 + 1$ such that $\abs{p_{n_2} - p} \ge \varepsilon$. Assume that we have chosen the first $k$ terms in the sequence, and then choose $n_{k+1} \ge n_k$ such that $\abs{p_{n_{k+1}} - p} \ge \varepsilon$.
This defines a subsequence $(p_{n_k})$ such that $\abs{p_{n_k}-p} \ge \varepsilon$ for all $k \in \mathbb{N}$. By assumption, there must exist a further subsequence $(p_{n_{k_\ell}})$ such that $p_{n_{k_\ell}} \to p$, but by construction all terms in any subsequence are at least a $\varepsilon$ distance from $p$, a contradiction.
\end{proof}

We can finally proceed in proving Theorem~\ref{thm:am_ext} below. 

\begin{proof}
Let $\{F_{n_k}\}_{k\in\mathbb{N}}$ be a subsequence of $\{F_n\}_{n\in\mathbb{N}}$. 
Since $F_n \to F$ in $L^p$, $F_{n_k} \to F$ in $L^p$ as well.
Thus, there exists a further subsequence $\{F_{n_{k_l}} \}_{l\in\mathbb{N}}$, that converges a.e.~to $F$. 
By Lemma \ref{lem:cont_ptz}, for any continuity point $\bm{x}$ of $F$, the subsubsequence $F_{n_{k_\ell}}(\bm{x})\to F(\bm{x})$. 
By Lemma \ref{lem:subsub}, $F_n(\bm{x}) \to F(\bm{x})$ at all continuity points of $F$. 
Finally, by Lemma \ref{lem:conv_dist}, $P_n \to P$ weakly.
\end{proof}

\section{Proof of Theorem \ref{thm:dci_binning}}
\label{app:dci_binning}

\begin{proof}

The proof of (i) follows immediately from Lemma~\ref{lem:optweights_converge}. 

To prove (ii), first we let $C_A:=Q^{-1}(Q(A))$ denote the contour event in $\pborel$ induced by the set $A$.
We apply the law of total probability to write
\begin{equation*}
    \initwemeaspart(A) = \initwemeaspart(A \, | \, C_A) \initwemeaspart(C_A) + \initwemeaspart(A \, | \, C_A^c) \initwemeaspart(C_A^c),
\end{equation*}
where $C_A^c$ denotes the complement of $C_A$. 
Since $A \subset C_A$, $A$ and $C_A^c$ are disjoint, which implies $\initwemeaspart(A\, | \, C_A^c)=0$.
Since $\initwemeaspart(C_A) = \predwemeaspart(Q(A))$, we have
\begin{equation*}
    \initwemeaspart(A) = \initwemeaspart(A \, | \, C_A) \predwemeaspart(Q(A)).
\end{equation*}

From (i), we have $\predwemeaspart(Q(A))\to \obsmeas(Q(A))>0$ as $n,p\to\infty$.
We use this along with the method of $n$ being chosen as a function of $p$ in Algorithm~\ref{alg:wEDFbinning} to observe that for sufficiently large $p$, $\predwemeaspart(Q(A))>0$. 
By the definition of $C_A$, it follows that both $\upmeas(C_A)>0$, and, for sufficiently large $p$, $\predwemeaspart(C_A)>0$.
Because the samples $\set{\bm{\lambda}^j}_{j=1}^n$ used to construct $\initwemeaspart$ are iid, the SLLN implies $\initwemeaspart(A\, | \, C_A) \to \initmeas(A\, | \, C_A)$.
Since $\upmeas$ and $\initmeas$ have the same conditional probabilities on contour events via the disintegration theorem, we have
\begin{equation*}
     \lim_{n,p\to\infty}\initwemeaspart(A\, | \, C_A) \to \\\initmeas(A\, | \, C_A) = \upmeas(A\, | \, C_A).
\end{equation*}
It follows that
\begin{equation*}
    \initwemeaspart(A) \to \upmeas(A\, | \, C_A) \obsmeas(Q(A)) \text{ as } n,p \to \infty.
\end{equation*}
The result follows by recognizing that $\obsmeas(Q(A)) = \upmeas(C_A)$ and applying the law of total probability to $\upmeas$ analogous to how it was applied to $\initwemeaspart$ above.
\end{proof}

\end{document}